\documentclass[12pt]{iopart}

\usepackage{amssymb}
\usepackage{amsthm}
\usepackage{esdiff}
\usepackage{color}
\usepackage{graphicx}
\usepackage{subfigure}
\usepackage{subeqn}
\usepackage{iopams}

\theoremstyle{plain}
\newtheorem{thm}{Theorem}[section]
\newtheorem{theorem}[thm]{Theorem}
\newtheorem{lemma}[thm]{Lemma}
\newtheorem{corollary}[thm]{Corollary}
\newtheorem{definition}[thm]{Defintion}
\newtheorem{ass}[thm]{Assumption}
\newtheorem{proposition}[thm]{Proposition}
\theoremstyle{remark}
\newtheorem{remark}[thm]{Remark}

\newcommand{\cc}{{\bf {\tiny \sc{pde}}}}
\newcommand{\cp}{{\bf {\tiny \sc{ppde}}}}
\newcommand{\um}{u^{\tiny {J}}}
\newcommand{\lef}{\ell^{\infty}(\IJ)}

\newcommand{\bbR}{\mathbb R}
\newcommand{\bbI}{\mathbb I}
\newcommand{\bbJ}{\mathbb J}
\newcommand{\cG}{{\mathcal G}}
\newcommand{\cF}{{\mathcal F}}
\newcommand{\cN}{{\mathcal N}}
\newcommand{\bbN}{\mathbb N}
\newcommand{\ba}{\overline {a}}

\newcommand{\amin}{a_{\tiny{\hbox{\sc min}}}}
\newcommand{\amax}{a_{\tiny{\hbox{\sc max}}}}
\newcommand{\aminb}{{\overline a}_{\tiny{\hbox{\sc min}}}}
\newcommand{\amaxb}{{\overline a}_{\tiny{\hbox{\sc max}}}}
\newcommand{\Li}{L^{\infty}(D)}

\newcommand{\Th}{\Theta}
\newcommand{\mud}{\mu^{\delta}}

\newcommand{\bbC}{{\mathbb C}}
\newcommand{\re}{\hbox{Re}}
\newcommand{\im}{\hbox{Im}}
\newcommand{\va}{\zeta}
\newcommand{\tp}{\tilde p}
\newcommand{\eps}{\epsilon}
\newcommand{\ys}{y^{\star}}
\newcommand{\zs}{z^{\star}}
\newcommand{\km}{\kappa_m}
\newcommand{\kj}{\kappa_j}
\newcommand{\kmp}{\kappa_m'}
\newcommand{\amp}{\amin'}
\newcommand{\kjp}{\kappa_j'}

\newcommand{\xr}{{\bf r}}

\definecolor{darkred}{rgb}{.7,0,0}

\newcommand{\beq}{\begin{equation}}
\newcommand{\eeq}{\end{equation}}
\newcommand{\bea}{\begin{eqnarray}}
\newcommand{\eea}{\end{eqnarray}}
\newcommand{\beas}{\begin{eqnarray*}}
\newcommand{\eeas}{\end{eqnarray*}}
\newcommand{\ds}{\displaystyle}

\def\IC{\mathbb{C}}
\def\IJ{\mathbb{J}}
\def\IN{\mathbb{N}}

\def\IL{\mathbb{P}}
\def\IR{\mathbb{R}}

\def\IE{\mathbb{E}}

\def\b1{{\bf 1}}
\def\cA{{\mathcal A}}
\def\cB{{\mathcal B}}

\def\cE{{\mathcal E}}
\def\cF{{\mathcal F}}
\def\cG{{\mathcal G}}

\def\cL{{\mathcal L}}
\def\cO{{\mathcal O}}
\def\cQ{{\mathcal Q}}
\def\cS{{\mathcal S}}
\def\cU{{\mathcal U}}

\newcommand{\eqref}[1]{$(\ref{#1})$}

\begin{document}

\title[Sparse Approximation of Inverse Problems]{Sparse Deterministic Approximation of Bayesian Inverse Problems}

\author{Ch. Schwab$^1$ and A.M. Stuart$^2$}
\address{$^1$ Seminar for Applied Mathematics,ETH, 8092 Zurich, Switzerland}
\address{$^2$ Mathematics Institute, University of Warwick, Coventry, CV4 7AL, UK}
\ead{
\\
christoph.schwab@sam.math.ethz.ch\\
a.m.Stuart@warwick.ac.uk 
}

\begin{abstract}
We present a parametric deterministic formulation
of Bayesian inverse problems with input parameter from
infinite dimensional, separable Banach spaces.
In this formulation,
the forward problems are parametric, deterministic
elliptic partial differential equations, and the
inverse problem is to determine the unknown,
parametric deterministic coefficients
from noisy
observations comprising linear functionals of the solution.

We prove a generalized polynomial chaos representation
of the posterior density with respect
to the prior measure, given noisy observational data.
We analyze the sparsity of the posterior density in terms
of the summability of the input data's
coefficient sequence. The first step in this process
is to estimate the fluctuations in the prior.
We exhibit sufficient conditions on the prior
model in order for approximations of the posterior
density to converge at a given algebraic rate, in terms of the
number $N$ of unknowns appearing in the parameteric
representation of the prior measure.
Similar sparsity and approximation results are also
exhibited for the solution and covariance of the
elliptic partial differential equation under the
posterior. These results then form the basis for
efficient uncertainty quantification, in the
presence of data with noise.
\end{abstract}


\section{Introduction} \label{sec:i}

Quantification of the uncertainty 
in predictions made by physical models,
resulting from uncertainty in the input parameters
to those models, is of increasing importance in 
many areas of science and engineering.
Considerable effort has been devoted to
developing numerical methods for this task.
The most straightforward approach is to sample 
the uncertain system responses by Monte Carlo
simulations. These have the advantage of being
conceptually straightforward, but are constrained
in terms of efficiency by their $N^{-\frac12}$ rate
of convergence ($N$ number of samples). 
In the 1980s
the engineering community started to develop new
approaches to the problem via parametric representation
of the probability space for the input parameters
\cite{SpG89,GSp03} based on the pioneering ideas of
Wiener \cite{Wiener1938}. 
The use of sparse spectral approximation techniques 
\cite{TS07,CS+CJG11} opens the avenue towards
algorithms for computational quantification of
uncertainty which beat the asymptotic complexity 
of Monte Carlo (MC) methods, as measured
by computational cost per unit error in predicted
uncertainty.

Most of the work in this area has been confined to
the use of probability models on the input parameters
which are very simple, albeit leading to high dimensional
parametric representations. Typically the randomness is
described by a (possibly countably infinite) set of
independent random variables representing 
uncertain coefficients in 
parametric expansions of input data, 
typically with known closed form Lebesgue densities. 
In many applications, such 
uncertainty in parameters is compensated for by 
(possibly noisy) observations, leading to an
inverse problem. 
One approach to such inverse
problems is via the techniques of optimal control \cite{BK89};
however this does not lead naturally to quantification
of uncertainty. A Bayesian approach to the
inverse problem \cite{ks04,Stuart10} allows the observations
to map a possibly simple prior probability distribution 
on the input parameters into a posterior distribution. 
This posterior distribution is typically much more complicated
than the prior, involving many correlations and 
without a useable closed form. 
The posterior distribution completely quantifies 
the uncertainty in the system's response, 
under given prior and structural assumptions on the system 
and given observational data. 
It allows, in particular, the Bayesian statistical
estimation of unknown system parameters and 
responses by integration with respect to the 
posterior measure, which is of interest in 
many applications.

Monte Carlo Markov chain (MCMC) methods can be used to
probe this posterior probability distribution. This
allows for computation of estimates of uncertain system responses
conditioned on given observation data by means of
approximate integration.
However, these methods suffer from the same limits on computational
complexity as straightforward Monte Carlo methods. It is
hence of interest to investigate whether sparse approximation
techniques can be used to approximate the posterior
density and conditional expectations given the data. 
In this pqper we study this question in the
context of a model elliptic inverse problem. 
Elliptic problems with random coefficients 
have provided an important class of model
problems for the uncertainty quantification community,
see, for example, \cite{BTZ04,CS+CJG11}
and the references therein. 
In the context of inverse problems and noisy observational data,
the corresponding elliptic problem arises naturally
in the study of groundwater flow (see \cite{McT96})
where hydrologists wish to determine the transmissivity 
(diffusion coefficient) from the head (solution of
the elliptic PDE). The elliptic inverse problem
hence provides natural model
problem within which to study sparse representations 
of the posterior distribution.

In Section \ref{sec:b} we recall the Bayesian
setting for inverse problems from \cite{Stuart10}, stating
and proving an infinite dimensional Bayes rule adapted
to our inverse problem setting in Theorem \ref{t:dens}.
Section \ref{sec:e} formulates the forward and
inverse elliptic 
problem of interest, culminating in an
application of Bayes rule in Theorem \ref{t:elldens}.
The prior model is built
on the work in \cite{BAS09,CDS1} in which the
diffusion coefficient is represented 
parametrically via an infinite
sum of functions, each with an independent uniformly
distributed and compactly supported random variable
as coefficient.  Once we have shown that the
posterior measure is well-defined and
absolutely continuous with respect to the prior, we
proceed to study the analytic dependence of the posterior
density in Section \ref{sec:comp}, culminating in 
Theorems \ref{thm:Analyt} and \ref{thm:PhiArbound}.
In Section \ref{sec:PCApprTh} we show how 
this parametric representation, and analyticity,
may be employed to
develop sparse polynomial chaos representations of the
posterior density, and the key 
Theorem \ref{semierror} summarizes the achievable
rates of convergence. 
In Section \ref{sec:Expost}
we study a variety of practical issues that arise
in attempting to exploit the sparse polynomial
representations as realizable algorithms for the
evaluation of (posterior) expectations.
Section \ref{sec:ConGen} contains our concluding remarks
and, in particular, a discussion of the computational
complexity of the new methodology, in comparison with
that for Monte Carlo based methods.
 
Throughout we concentrate on the posterior density
itself. However we also provide analysis related to
the analyticity (and hence sparse polynomial
representation) of various functions of the unknown
input, in particular the solution to the forward
elliptic problem, and tensor products of this function.
For the above class of elliptic model problems,
we prove that for given data, there exist sparse,
$N$-term gpc (``generalized polynomial chaos'') 
approximations of this expectation with respect
to the posterior
(which is written as a density reweighted
expectation with respect to the prior) 
which converge at the same rates afforded
by best $N$-term gpc approximations of the 
system response to uncertain, parametric inputs.
Moreover, our analysis implies that the set
$\Lambda_N$ of the $N$ ``active'' gpc-coefficients
is identical to the set $\Lambda_N$ 
of indices of a best $N$-term approximation of 
the system's response.
It was shown in \cite{CDS1,CDS2} that these rates
are, in turn, completely determined by the 
the decay rates of the input's fluctuation 
expansions. We thus show that the machinery
developed to describe gpc approximations of
uncertain system response may be employed to study
the more involved Bayesian inverse problem where
the uncertainty is conditioned on observational data. 
Numerical algorithms which achieve the optimal
complexity implied by the sparse approximations,
and numerical results demonstrating this 
will be given in our forthcoming work \cite{AScSt}.

\section{Bayesian Inverse Problems} \label{sec:b}

%
Let $G: X \to R$ denote a ``forward'' map from 
some separable Banach space $X$ of unknown parameters 
into another separable Banach space $R$ of responses. 
We equip $X$ and $R$ with norms $\| \cdot\|_X$ 
and  with $\| \cdot \|_R$, respectively.
In addition, we are given 
$\cO(\cdot): R \rightarrow \IR^K$ denoting a bounded
linear observation operator on the space $R$ of
system responses, which belong 
to the dual space $R^*$ 
of the space $R$ of system responses.
We assume that the data is finite
so that  $K<\infty,$ 
and equip $\IR^K$ with the Euclidean norm, denoted by $|\cdot|.$ 

We wish to determine the unknown data 
$u \in X$ from the noisy observations
\begin{equation}
\label{eq:obs1}
\delta = \cO(G(u)) + \eta
\end{equation}
where $\eta \in \IR^K$ represents the noise. We
assume that realization of the noise process
is not known to us, but that it is a draw from
the Gaussian measure $\cN(0,\Gamma),$ for some positive
(known) covariance operator $\Gamma$ on $\IR^K$. 
If we define $\cG:X \to \IR^K$ by $\cG=\cO \circ G$ then
we may write the equation for the observations as
\begin{equation}
\label{eq:obs11}
\delta = \cG(u) + \eta. 
\end{equation}
We define the least squares functional 
(also referred to as ``potential'' in what follows)
$\Phi:X \times \IR^K\to \bbR$ 
by
\begin{equation}
\Phi(u;\delta) = \frac12|\delta - \cG(u) |_{\Gamma}^2
\label{eq:lsq}
\end{equation}
where $|\cdot|_{\Gamma} = |\Gamma^{-\frac12}\cdot|$ 
so that 
$$
\Phi(u;\delta)
=
\frac12
\left(
(\delta - \cG(u))^\top \Gamma^{-1} (\delta - \cG(u))
\right)
\;.
$$
%

In \cite{Stuart10} it is shown that, under appropriate conditions
on the forward and observation model $\cG$ and the prior
measure on $u$, the posterior
distribution on $u$ is absolutely continuous with
respect to the prior with Radon-Nikodym derivative
given by an infinite dimensional version of Bayes rule.
Posterior uncertainty is then
determined by integration of suitably chosen functions 
against this posterior. 
At the heart of the deterministic approach proposed 
and analyzed here lies the {\em
reformulation of the forward problem with stochastic
input data as an infinite dimensional, parametric 
deterministic problem}.
We are thus interested in expressing 
the posterior distribution in terms of a parametric
representation of the unknown coefficient function $u$.
To this end we assume 
that, under the prior distribution,
this function admits a 
{\em parametric representation} 
of the form
\begin{equation}
\label{eq:par}
u=\bar{a} + \sum_{j \in \bbJ} y_j\psi_j
\end{equation}
where $y=\{y_j\}_{j \in \bbJ}$ is an
i.i.d sequence of real-valued random variables
$y_j \sim \cU(-1,1)$ and $\bar{a}$ and the $\psi_j$ are
elements of $X$. Here and throughout, $\bbJ$ denotes a 
finite or countably infinite index set, i.e. either
$\bbJ = \{1,2,...,J\}$ or $\bbJ = \bbN$.
All assertions proved in the present paper hold in either
case, and all bounds are in particular independent of 
the number $J$ of parameters.

To derive the parametric expression of the prior measure $\mu_0$ 
on $y$ we denote by 
$$
U=(-1,1)^\bbJ
$$
the space of all sequences 
$(y_j)_{j\in\bbJ}$ of real numbers $y_j\in (-1,1)$.
Denoting the sub $\sigma$-algebra of 
Borel subsets on $\mathbb{R}$ which are also subsets
of $(-1,1)$ by $\cB^1(-1,1)$, the pair 
\beq\label{eq:UcB}
(U,\cB) 
= 
\left((-1,1)^\bbJ,\; \bigotimes_{j\in\bbJ} \cB^1(-1,1) \right)
\eeq
is a measurable space. 
We equip $(U,\cB)$ with the uniform probability 
measure
\beq\label{eq:defmu0}
\mu_0(dy) := \bigotimes_{j\in \IJ} \frac{dy_j}{2}
\eeq
which corresponds to bounded intervals for the 
possibly countably many uncertain parameters. 
Since the countable product of probability measures is 
again a probability measure, $(U,\cB,\mu_0)$ is a probability
space.
{\em We assume in what follows that the prior measure on the
uncertain input data, parametrized in the form \eqref{eq:par}, 
is $\mu_0(dy)$.}
We add in passing that unbounded parameter ranges
as arise, e.g., in lognormal random 
diffusion coefficients in models for subsurface flow
\cite{McT96},
can be treated by the techniques developed here, 
at the expense of additional technicalities. 
We refer to \cite{AScSt} for details as well as for numerical
experiments.

Define $\Xi:U \to \IR^K$ by
\begin{equation} \label{eq:ymap}
\Xi(y)=\cG(u)\Bigl|_{u=\bar{a}+\sum_{j \in \bbJ} y_j\psi_j}.
\end{equation}
In the following we view $U$ as  a bounded subset in
$\lef$, the Banach space of bounded sequences, 
and thereby introduce a notion of continuity in $U$.

\begin{theorem} \label{t:dens}
Assume that $\Xi:{\bar U} \to \IR^K$ is bounded
and continuous. Then $\mud(dy)$, the distribution of
$y$ given $\delta$, is
absolutely continuous with respect to $\mu_0(dy)$.
Furthermore, if  
\begin{equation} \label{eq:PostDens}
\Th(y)=\exp\bigl(-\Phi(u;\delta)\bigr)\Bigl|_{u=\bar{a}+\sum_{j \in \bbJ} y_j\psi_j},
\end{equation}
then
\beq \label{eq:post}
\frac{d\mud}{d\mu_0}(y) = \frac{1}{Z} \Th(y)
\eeq
where
\beq \label{eq:Z}
Z=\int_{U} \Th(y)\mu_0(dy).
\eeq
\end{theorem}

\begin{proof}
Let $\nu_0$ denote the probability measure 
on $U\times \IR^K$ defined by $\mu_0(dy) \otimes
\pi(d\delta),$ where $\pi$ is the Gaussian measure 
$\cN(0,\Gamma).$ Now define a second probability measure
$\nu$ on $U \times \IR^K$ as follows. 
First we specify the distribution of $\delta$ given $y$
to be $\cN(\Xi(y),\Gamma)$. 
Since $\Xi(y):{\bar U} \to \IR^K$ is continuous and $\mu_0(U)=1$
we deduce that $\Xi$ is $\mu_0$ measurable. Hence
we may complete the definition of $\nu$ by specifying 
that $y$ is distributed according to $\mu_0$.
By construction, and ignoring the constant of
proportionality which depends only on $\delta$, 
\footnote{$\Theta(y)$ is also a function of
$\delta$ but we suppress this for economy of notation.
}
$$
\frac{d\nu}{d\nu_0}(y,\delta) \propto  \Th(y).$$
From the boundedness of $\Xi$ on ${\bar U}$ we deduce 
that $\Th$ is bounded from below on ${\bar U}$ by $\theta_0>0$
and hence that 
$$Z \ge \int_{U} \theta_0 \mu_0(dy)=\theta_0>0$$ 
since $\mu_0(U)=1$.
Noting that, under $\nu_0$, $y$ and $\delta$ are independent,
Lemma 5.3 in \cite{HSV05} gives the desired result.
\end{proof}
We assume that we wish to 
compute the expectation of 
a function $\phi:X \to S$, for some Banach space $S$.
With $\phi$, we associate the parametric mapping
\begin{equation}
\Psi(y)
=
\exp\bigl(-\Phi(u;\delta)\bigr)\phi(u)\Bigl|_{u={\bar a}
+
\sum_{j \in \bbJ} y_j\psi_j}
: U\rightarrow S 
\;.
\label{eq:psi}
\end{equation}
From $\Psi$ we define
\begin{equation}\label{eq:intpsi}
Z'=\int_{U} \Psi(y)\mu_0(dy) \in S
\end{equation}
so that the expectation of interest is given by $Z'/Z\in S$.
Thus our aim is to approximate $Z'$ and $Z.$ 
Typical choices for $\phi$ in applications might
be $\phi(u)=G(u)$, the response of the system,
or 
\beq\label{eq:defmpointcorp}
\phi(u):=\bigl(G(u)\bigr)^{(m)} := \underbrace{G(u) \otimes ... \otimes G(u)}_{m \;\mbox{times}} 
      \in S = R^{(m)} := \underbrace{R\otimes ... \otimes R}_{m \;\mbox{times}}
\;.
\eeq
In particular the choices $\phi(u)=G(u)$
and $\phi(u)=G(u)\otimes G(u)$ together 
facilitate computation of the mean and 
covariance of the response.

In the next sections we will study the elliptic problem
and deduce, from known results concerning the 
parametric forward problem, the joint analyticity
of the posterior density $\Th(y)$, and also
$\Psi(y)$, as a function of the parameter vector $y\in U$. 
From these results, we deduce 
{\em 
sharp estimates on size of domain of analyticity
of $\Th(y)$ (and $\Psi(y)$)
as a function of each coordinate $y_j$, $j\in\mathbb{N}$.
}
We concentrate on the concrete choice of $\Psi$ defined
by \eqref{eq:defmpointcorp}, and often the
case $p=1$. The analysis can be extended to other
choices of $\Psi$.

\section{Model Parametric Elliptic Problem} \label{sec:e}
\subsection{Function Spaces} \label{sssec:e1a}
Our aim is to
study the inverse problem of determining the diffusion
coefficient $u$ of an elliptic PDE from observation of
a finite set of noisy linear functionals of the solution $p$,
given $u$.

Let $D$ be a bounded Lipschitz domain in $\bbR^d$, 
$d=1,2$ or $3$, with Lipschitz boundary $\partial D$.
Let further $\Bigl(H,(\cdot,\cdot),\|\cdot\|\Bigr)$
denote the Hilbert space $L^2(D)$ which we will identify
throughout with its dual space, i.e.  $H \simeq H^*$.

We define also the space $V$ of variational solutions
of the forward problem: specifically, we let 
$\Bigl(V,(\nabla\cdot,\nabla\cdot),\|\cdot\|_{V}\Bigr)$
denote the Hilbert space $H^1_0(D)$ 
(everything that follows will hold for rather general,
 elliptic problems with affine parameter dependence
 and ``energy'' space $V$).
The dual space $V^*$ of all continuous, linear functionals 
on $V$ is isomorphic to the Banach space $H^{-1}(D)$ 
which we equip with the dual norm to $V$, 
denoted $\|\cdot\|_{-1}$.
We shall assume for the (deterministic) data
$f\in V^*$. 
\subsection{Forward Problem} \label{sssec:se1c}
In the bounded Lipschitz domain $D$, 
we consider the following elliptic PDE:
\beq
-\nabla \cdot\bigl(u\nabla p\bigr) = f \quad \mbox{in}\quad D,
\qquad
p=0 \quad \mbox{in} \quad \partial D.
\label{eq:fwdproblem}
\eeq
Given data $u\in  L^\infty(D)$, a 
weak solution of \eqref{eq:fwdproblem} 
for any $f\in V^*$ is a function 
$p\in V$ which satisfies  
\beq
\label{weaksol}
\int_D  u (x) \nabla p(x)\cdot \nabla q(x) dx
=
_V\langle q,f \rangle_{V^*} 
\;\; \mbox{for all}\;\; q\in V\;.
\eeq
Here $_V\langle \cdot,\cdot \rangle_{V^*}$ denotes
the dual pairing between elements of $V$ and $V^*.$

For the well-posedness of the forward problem, 
we shall work under 
\begin{ass}\label{assump1}
There exist constants $0<\amin \le \amax < \infty$
so that
\beq \label{cLM1real}
0 < \amin \le u(x) \le \amax < \infty,\quad  x\in D,
\eeq
\end{ass}
Under Assumption \ref{assump1},
the Lax-Milgram Lemma ensures the existence and 
uniqueness of the response $p$ of \eqref{weaksol}. Thus,
in the notation of the previous section, $R = V$
and $G(u)=p.$ 
Moreover, this variational solution satisfies the a-priori 
estimate  
\beq
\label{cLM2}
\|G(u)\|_V = \|p \|_V \le \frac{\|f\|_{V^*}}{\amin}\;. 
\eeq 
We assume that the observation function $\cO:V \to \bbR^K$
comprises $K$ linear functionals $o_k \in V^*$, $k=1,\dots,K$.
In the notation of the previous section,
we denote by $X=L^\infty(D)$ the Banach space in which the
unknown input parameter $u$ takes values. It follows that
\beq
\label{cLM22}
|\cG(u)| \le \frac{\|f\|_{V^*}}{\amin}
\bigl(\sum_{k=1}^K\|o_k\|_{V^*}^2\bigr)^{\frac12}\;.
\eeq

\subsection{Structural Assumptions on Diffusion Coefficient} 
As discussed in section \ref{sec:b} we introduce a 
parametric representation of the random input parameter $u$
via an affine representation with respect to $y$,
which means that the parameters $y_j$ are the coefficients 
of the function $u$ in the formal series expansion
\beq
\label{assume1}
u(x,y)=\bar a(x)+ \sum_{j\in \IJ} y_j\psi_j(x),\quad x\in D,
\eeq
where $\bar a\in L^\infty(D)$ and $\{\psi_j\}_{j\in \IJ}\subset L^\infty(D)$. 
We are interested in the effect of approximating the
solutions input parameter $u(x,y),$ by
truncation of the series expansion \eqref{assume1} in the
case $\IJ=\IN$,
and on the corresponding effect on the forward (resp.
observational) map $G(u(\cdot))$ (resp. $\cG(u(\cdot))$)
to the family of elliptic equations with the 
above input parameters. 
In the decomposition \eqref{assume1}, we have the choice to either 
normalize the basis (e.g., assume they all have norm one in some space) 
or to normalize the parameters.  It is more convenient for us to do   
the latter.
This leads us to the following assumptions which shall be made throughout:
\begin{itemize}
\item[i)] 
For all $j\in \bbJ: \psi_j \in L^\infty(D)$ and $\psi_j(x)$ is 
defined for all $x\in D$,
\item[ii)] 
\beq \label{eq:defU}
y=(y_1,y_2,\dots) \in U=[-1,1]^{\bbJ},
\eeq
i.e. the parameter vector $y$ in \eqref{assume1}
belongs to the unit ball of the 
sequence space $\ell^\infty(\bbJ)$,
\item[iii)] for each $u(x,y)$ to be considered, 
\eqref{assume1} holds 
for every $x\in D$ and every $y\in U$.
\end{itemize}
We will, on occasion, use \eqref{assume1}
with $\IJ \subset \IN$, as well as with $\IJ=\IN$ 
(in the latter case the additional Assumption \ref{assump2} 
 below has to be imposed).
In either case, we will work throughout under the assumption that the ellipticity
condition \eqref{cLM1real} holds uniformly for $y\in U$.

\noindent
{\bf Uniform Ellipticity Assumption:}  
{\it there exist $0 < \amin \le \amax < \infty$ 
such that for all $x\in D$ and for all $y\in U$
\beq \label{primary}
0< \amin \le u(x,y) \le \amax < \infty.
\eeq
}

We refer to assumption (\ref{primary}) 
as ${\bf UEA}(\amin,\amax)$ in the following.
In particular, ${\bf UEA}(\amin,\amax)$ implies
$\amin \le \bar a(x) \le \amax$ for all $x\in D$, 
since we can choose $y_j=0$ for all $j\in\mathbb{N}$.
Also observe that the validity of the lower and upper 
inequality in \eqref{primary} 
for all $y\in U$ are respectively equivalent 
to the conditions that
\beq
\label{primary1}
\sum_{j\in \IJ}|\psi_j(x)| \le  \bar a(x) - \amin , \quad x\in D,
\eeq
and 
\beq
\label{primary1up}
\sum_{j\in \IJ}|\psi_j(x)|\le  \amax - \bar a(x),\quad x\in D.
\eeq
We shall require in what follows a quantitative control of 
the relative size of the fluctuations in the representation
\eqref{assume1}. To this end, we shall impose
\begin{ass}\label{assump2}
The functions $\bar a$ and $\psi_j$ in \eqref{assume1} satisfy
\[
\sum_{j\in \IJ}\|\psi_j\|_{L^\infty(D)}\le {\kappa\over 1+\kappa}\aminb,
\]
with $\aminb = \min_{x\in D}\bar a(x)>0$ and $\kappa>0$.
\end{ass}
Assumption \ref{assump1} is then satisfied by choosing
\beq
\amin := \aminb - {\kappa\over 1+\kappa}\aminb = {1\over 1+\kappa}\aminb.
\label{amin}
\eeq
\subsection{Inverse Problem}
We start by proving that the forward 
maps $G:X \to V$ and $\cG:X \to \IR^K$ are Lipschitz. 
 
\begin{lemma}
\label{lemmastab}
If $p$ and $\tilde{p}$ are solutions of \eqref{weaksol}
with the same right hand side $f$ and 
with coefficients $u$ and $\tilde{u}$, 
respectively,
and if these coefficients both satisfy Assumption \ref{assump1}
then the forward solution map $u\rightarrow p=G(u)$ is Lipschitz
as a mapping from $X$ into $V$ with Lipschitz constant defined by
\beq
\|p-\tilde{p}\|_V 
\leq 
\frac{\|f\|_{V^*}} {\amin^2} \|u - \tilde{u}\|_{L^\infty(D)}.
\label{stab}
\eeq
Moreover
the forward solution map can be composed with the
observation operator to prove that  the map 
$u\rightarrow \cG(u)$ is Lipschitz as a mapping from $X$ into $\IR^K$
with Lipschitz constant defined by 
\beq\label{eq:LipcG}
| \cG(u) - \cG(\tilde{u}) | 
\le 
\frac{\|f\|_{V^*}} {\amin^2} 
\bigl(\sum_{k=1}^K\| o_k \|_{V^*}^2\bigr)^{\frac12}
\|u - \tilde{u}\|_{L^\infty(D)}.
\eeq
\end{lemma}
{\em Proof}:
Subtracting the variational formulations for $p$ and $\tilde{p}$,
we find that for all $q\in V$,
$$
0=\int_D u \nabla p\cdot \nabla q dx 
 -
 \int_D \tilde{u}\nabla \tilde{p}\cdot \nabla q dx
 =
 \int_D u(\nabla p-\nabla \tilde{p}) \cdot \nabla q dx
 +
 \int_D (u - \tilde{u}) \nabla \tilde{p}\cdot \nabla q dx.
$$
Therefore $w=p - \tilde{p}$ 
is the solution of
$\int_D u \nabla w\cdot \nabla q=L(q)$
where
$L(v):=\int_D (\tilde{u}-u) \nabla \tilde{p}\cdot \nabla v$.
Hence
$$
\|w\|_V\leq \frac {\|L\|_{V^*}} \amin,
$$
and we obtain \eqref{stab} since it follows from \eqref{cLM2}
that
$$
\|L\|_{V^*}
=
\max_{\|v\|_V=1}|L(v)| 
\leq 
\|u -\tilde{u}\|_{L^\infty(D)}\|\tilde{p}\|_V
\leq 
\|u -\tilde{u}\|_{L^\infty(D)} \frac{ \|f\|_{V^*}}{\amin}.
$$
Lipschitz continuity of $\cG=\cO\circ G: X \to \IR^K$
is immediate since $\cO$ comprises the $K$
linear functionals $o_k$.
Thus \eqref{stab} implies \eqref{eq:LipcG}.
\hfill $\Box$

The next result may be deduced in a straightforward
fashion from the preceding analysis:

\begin{theorem} 
\label{t:elldens}
Under the ${\bf UEA}(\amin,\amax)$ and
Assumption \ref{assump2} it follows that the posterior
measure $\mud(dy)$ on $y$ given $\delta$ is absolutely
continuous with respect to the prior measure $\mu_0(dy)$
with Radon-Nikodym derivative given by \eqref{eq:PostDens}
and \eqref{eq:post}.
\end{theorem}

\begin{proof} 
This is a straightforward consequence of Theorem \ref{t:dens}
provided that we show boundedness
and continuity of $\Xi:{\bar U} \to \bbR^K$ given by 
\eqref{eq:ymap}.
Boundedness follows from \eqref{cLM22}, together with
the boundedness of $\|o_k\|_{V^*}$, under ${\bf UEA}(\amin,\amax)$.
Let $u,\tilde{u}$ denote two diffusion coefficients
generated by two parametric sequences $y, \tilde{y}$ in $U$.
Then, by \eqref{eq:LipcG} and Assumption \ref{assump2}, 
\begin{eqnarray*}
| \Xi(y) - \Xi(\tilde{y}) | 
&\le 
 \frac{\|f\|_{V^*}} {\amin^2} \bigl(\sum_{k=1}^K\|o_k\|_{V^*}^{2}\bigr)^{\frac12}
\|u - \tilde{u}\|_{L^\infty(D)}
\\
&\le 
\frac{\|f\|_{V^*}} {\amin^2} \bigl(\sum_{k=1}^K\|o_k\|_{V^*}^{2}\bigr)^{\frac12}
\frac{\kappa}{1+\kappa}{\aminb}\|y-\tilde{y}\|_{\lef}
\;.
\end{eqnarray*}
The result follows.
\end{proof}

\section{Complex Extension of the Elliptic Problem}
\label{sec:comp}
As indicated above, 
one main technical objective will consist in
proving analyticity of the posterior density $\Th(y)$ with
respect to the (possibly countably many) parameters
$y\in U$ in \eqref{assume1} defining the prior, 
and to obtain bounds 
on the supremum of $\Theta$ over the
maximal domains in $\IC$ into which $\Th(y)$ can be 
continued analytically.
Our key ingredients for getting such estimates rely on complex analysis.

It is well-known that 
the existence theory for the forward problem
\eqref{eq:fwdproblem}
extends to the case where the coefficient function 
$u(x)$ takes values in $\IC$.
In this case, the ellipticity Assumption \ref{assump1} 
should be replaced by the assumption that
\beq \label{cLM1}
0< \amin \le \Re(u(x))\le |u(x)|\le \amax  < \infty,\quad  x\in D.
\eeq
and all the above results remain valid with 
Sobolev spaces understood as spaces of 
complex valued functions.
Throughout what follows, we shall frequently 
pass to spaces of complex valued functions,
without distinguishing these notationally.
It will always be clear from the context which
coefficient field is implied.

\subsection{Notation and Assumptions}
We extend the definition of $u(x,y)$ to $u(x,z)$ 
for the complex variable $z=(z_j)_{j\in \bbJ}$ 
(by using the $z_j$ instead of $y_j$ in
 the definition of $u$ by \eqref{assume1}) 
where each $z_j$ has modulus less than or equal to $1$. 
Therefore $z$ belongs to the polydisc
\beq \label{eq:defcU}
\cU := \bigotimes_{j\in \IJ}\{z_j\in \mathbb{C}: |z_j|\leq 1\}
\subset \mathbb{C}^{\mathbb{J}} \;.
\eeq
Note that $\overline{U} \subset \cU$.
Using \eqref{primary1} and \eqref{primary1up}, 
when the functions 
$\overline{a}$ and $\psi_j$ are real valued, 
condition ${\bf UEA}(\amin,\amax)$ 
implies that for all $x\in D$ and $z\in \cU$,
\beq \label{eq:UEAC2R}
0< \amin \leq \Re(u(x,z)) \leq |u(x,z)| \leq 2\amax \;, 
\eeq
and therefore the corresponding solution $p(z)$ 
is well defined in $V$ for all $z\in\cU$
by the Lax-Milgram theorem for sesquilinear forms.
More generally, we may consider an expansion of the form,
$$
u(x,z)=\overline{a} +\sum_{j\in \IJ} z_j\psi_j
$$
where $\overline{a}$ and $\psi_j$ are 
complex valued functions
and replace ${\bf UEA}(\amin ,\amax )$ 
by the following, complex-valued counterpart:
\newline
{\bf Uniform Ellipticity Assumption in $\IC$} : 
{\it 
there exist $0<\amin \le \amax <\infty$ 
such that for all $x\in D$ and all $z\in \cU$
\beq \label{primarycomp}
0< \amin  \le \Re(u(x,z)) \le |u(x,z)|\leq  \amax  < \infty.
\eeq
}
We refer to (\ref{primarycomp}) as ${\bf UEAC}(\amin,\amax)$.
\subsection{Domains of holomorphy}

The condition ${\bf UEAC}(\amin,\amax)$ implies that 
the forward solution map $z\mapsto p(z)$ is 
strongly holomorphic as a $V-$valued function 
which is uniformly bounded
in certain domains larger than $\cU$.  
For $0< r \leq 2\amax < \infty$ 
we define the open set 
\beq
\label{eq:Adelta}
\cA_r
= 
\{ z\in  \IC^{\IJ}: r < \Re(u(x,z)) \le |u(x,z)| < 2\amax 
\quad\mbox{for every} \quad x\in D
\} \subset \IC^{\IJ}
\;.
\eeq
Under ${\bf UEAC}(\amin , \amax )$, 
for every
$0 < r < \amin $ holds $\cU \subset \cA_r$.

According to the Lax-Milgram theorem, 
for every $z\in \cA_r$
there exists a unique solution $p(z)\in V$ 
of the variational problem: 
given $f\in V^*$, 
for every $z\in \cA_r$, 
find $p\in V$ such that 
\beq\label{eq:sesqfwdpb}
\alpha(z;p,q) = (f,q) \qquad \forall q\in V \;.
\eeq
Here the sesquilinear form $\alpha(z;\cdot,\cdot)$
is defined as
\begin{equation} \label{eq:alv1}
\alpha\bigl(z;p,q) = \int_D u(x,z) \nabla p \cdot \overline{\nabla q} dx
\quad \forall p,q \in V
\;.
\end{equation}
We next show that the analytic continuation 
of the parametric solution $p(y)$ to the domain $\cA_r$
is the unique solution $p(z)$ 
of \eqref{eq:sesqfwdpb} which satisfies 
the a-priori estimate 
\beq
\sup_{z\in \cA_r} \|p(z)\|_V \le \frac{\|f\|_{V^*}}{r}.
\label{eq:zApriori}
\eeq
The first step of our analysis is
to establish strong holomorphy
of the forward solution map $z\mapsto p(z)$ 
in \eqref{eq:sesqfwdpb}
with respect 
to the countably many variables $z_j$
at any point $z\in \cA_r$.
This follows from the observation
that the function $p(z)$ is the solution to 
the operator equation $A(z)p(z)=f$, where
the operator $A(z)\in \cL(V,V^*)$ depends in
an affine manner on each variable $z_j$.
To prepare the argument for proving holomorphy 
of the functionals $\Phi$ and $\Th$ appearing in 
\eqref{eq:PostDens}, \eqref{eq:psi}
we give a direct proof.

Using Lemma \ref{lemmastab} we have proved by means
of a difference quotient argument given in \cite{CDS2}, 
Lemma \ref{lem:VderivC} ahead.
Lemma  \ref{lem:VderivC}, together with Hartogs' Theorem 
(see, e.g., \cite{HoermandComplex})
and the separability of $V$, implies strong
holomorphy of $p(z)$ as a $V$-valued function on
$\cA_r$, stated as Theorem \ref{thm:Analyt}
below.
The proof of this theorem can also be found in \cite{CDS2};
the result will also be obtained as a corollary of 
the analyticity results for the functionals
$\Psi$, $\Th$ proved below.
\begin{lemma}\label{lem:VderivC}
At any $z\in \cA_r$, the function $z\mapsto p(z)$ admits a 
complex derivative
$\partial_{z_j}p(z)\in V$ with respect to each variable $z_j$. 
This derivative is the weak solution of the problem: 
given $z\in  \cA_r$, find $\partial_{z_j}p(z)\in V$ 
such that
\beq\label{eq:partuzj}
\alpha(z; \partial_{z_j}p(z) , q  )
=
L_0(q) 
:= 
-\int_D \psi_j\nabla p(z) \cdot \overline{\nabla q} dx \;,
\qquad
\mbox{for all } q\in V.
\eeq
\label{lemmahol}
\end{lemma}

\begin{theorem}\label{thm:Analyt}
Under ${\bf UEAC}(\amin,\amax)$
for any $0 < r < \amin$ the solution $p(z)=G(u(z))$ 
of the parametric forward problem
is holomorphic as a $V$-valued function in $\cA_r$
and the {\em a priori} estimate \eqref{eq:zApriori} holds.
\end{theorem}

We remark that $\cA_r$ also contains certain {\it polydiscs}:
for any sequence $\rho:=(\rho_j)_{j\geq 1}$ of positive 
radii we define the polydisc
\beq
\cU_\rho
=
\bigotimes_{j\in \IJ}\{z_j \in \mathbb{C}: |z_j|\leq \rho_j\}
=
\{z_j\in \mathbb{C}: z=(z_j)_{j\in \IJ}\,\; ; \;\; |z_j|\leq \rho_j\}
\subset \IC^\IJ \;.
\label{polydisc}
\eeq
We say that {\em a sequence $\rho = (\rho_j)_{j \geq 1}$ 
of radii is $r$-admissible} if and only if
for every $x\in D$
\beq
\label{deltaadmissible}
\sum_{j\in \IJ} \rho_j |\psi_j(x)| \le \Re(\bar a(x)) - r.
\eeq
If the sequence $\rho$ is $r$-admissible, 
then the polydisc $\cU_\rho$ is contained in $\cA_r$
since on the one hand for all $z\in \cU_\rho$
and for almost every $x\in D$
$$
\Re(\bar u(x,z)) \geq \Re(\bar a(x))-\sum_{{j\in\IJ}} |z_j \psi_j(x)| 
\\
\geq  \Re(\bar a(x))-\sum_{{j\in \IJ}} \rho_j | \psi_j(x)|\geq r,
$$
and on the other hand, if for every $x\in D$
$$
|u(x,z)|
\leq |\bar a(x)| + \sum_{j\in \IJ} |z_j \psi_j(x)|
\leq |\bar a(x)| + \Re(\bar a(x)) - r \leq 2|\bar a(x)| \leq 2 \amax \;.
$$
Here we used $|\bar a(x)| \leq \amax$ which follows 
from ${\bf UEAC}(\amin , \amax )$.

Similar to \eqref{primary1}, 
the validity of the lower inequality in 
\eqref{primarycomp} for all $z\in \cU$ is equivalent to the
condition that
\beq
\label{primary1comp}
\sum_{j\geq 1}|\psi_j(x)|
\le  
\Re(\bar a(x)) - \amin,\quad x\in D.
\eeq
This shows that the constant sequence $\rho_j=1$ is $r$-admissible 
for all $0< r \leq \amin $.
\begin{remark}\label{remk:rho>1}
For $0<r<\amin$ there exist $r$-admissible sequences
such that $\rho_j>1$ for all $j\geq 1$, i.e. such that the polydisc 
$\cU_\rho$ is strictly larger than $\cU$ in every variable.
This will be exploited systematically below in the derivation
of approximation bounds.
\qed
\end{remark}

\subsection{Holomorphy of response functionals}

We next show that, for given data $\delta$,
the functionals 
$\cG(\cdot)$, $\Phi(u(\cdot);\delta)$ and $\Th(\cdot)$ 
depend holomorphically on the parameter vector $z\in\IC^\IJ$,
on polydiscs $\cU_\rho$ as in \eqref{polydisc} for suitable 
$r$-admissible sequences of semiaxes $\rho$. 
Our general strategy for proving this will be analogous 
to the argument for establishing analyticity of the map 
$z \mapsto G(u(z))$ as a $V$-valued function.

We now extend Theorem \ref{thm:Analyt} from the
solution of the elliptic PDE to the posterior density,
and related quantities required to define expectations
under the posterior, culminating in Theorem \ref{thm:PhiArbound}
and Corollary \ref{cor:PsiAnalyt}. We achieve this
through a sequence of lemmas which we now derive.

The following lemma is simply a complexification of
\eqref{cLM22} and \eqref{eq:LipcG}.
It implies bounds on $\cG$ and its
Lipschitz constant in the covariance weighted norm.
\begin{lemma}
Under  ${\bf UEAC}(\amin,\amax)$, 
for every $f\in V^* = H^{-1}(D)$ and for every 
$\cO(\cdot) \in (V^*)^* \simeq V \rightarrow Y = \IR^K$
holds
\begin{eqnarray}
|\cG(u)| &\le  \frac{\|f\|_{V^*}}{\amin} 
\bigl(\sum_{k=1}^K\|o_k\|_{V^*}^2\bigr)^{\frac12}
\; ,
\\
|\cG(u)-\cG(u)| 
& \le 
 \frac{\|f\|_{V^*}}{\amin^2}\|u_1-u_2\|_{\Li}
\bigl(\sum_{k=1}^K\|o_k\|_{V^*}^2\bigr)^{\frac12}.
\end{eqnarray}
\end{lemma}
To be concrete we concentrate in the next
lemma on computing the expected
value of {\em the pressure} $p=G(u) \in V$
under the posterior measure. To this end
we define $\Psi$ with $\psi$ as in \eqref{eq:defmpointcorp} 
with $m=1$.
We start by considering the case of a single parameter.

\begin{lemma} \label{l:one}
Let $\IJ=\{1 \}$ and take $\phi = G: U \to V$.
With $u(x,y)$ as in \eqref{eq:par},
under  ${\bf UEAC}(\amin,\amax)$, the functions
$\Psi:[-1,1] \to V$ and $\Theta:[-1,1] \to \bbR$ 
and the potential $\Phi(u(x,\cdot);\delta)$
defined by \eqref{eq:psi}, \eqref{eq:PostDens}
and \eqref{eq:lsq} respectively,
may be extended to functions which are strongly 
holomorphic on the strip $\{y+iz: |y|<r/\kappa\}$ 
for any $r \in (\kappa,1)$.
\end{lemma}

\begin{proof} 
We view $H,V$ and $X=L^{\infty}(D)$ as Banach spaces over $\bbC$. 
We extend the equation \eqref{assume1} to complex
coefficients $u(x,z) = \re(\ba(x) + z\psi(x)) = \ba(x) + y\psi(x)$ 
since $z=y+i\va$.
Note that $\ba+z\psi$ is holomorphic in $z$ since it is linear.
Since $\re(\ba+z\psi)=\ba+y\psi \ge \amin$, 
if follows that, for all $\va=\im(z)$,
$$
\re \int_{D} u(x)|\nabla p(x)-\nabla \tp(x)|^2 dx \ge \amin\|p-\tp\|_{V}^2.
$$

We prove that the mappings $\Psi$ and $\Theta$
are holomorphic by studying the properties of
$G(\ba+z\psi)$ and $\Phi(\ba+z\psi)$ as functions
of $z \in \bbC.$ 
Let $h \in \bbC$ with $|h|<\eps \ll 1$\;.
We show that
$$\lim_{|h| \to 0} h^{-1}\bigl(p(z+h)-p(z)\bigr)$$
exists in $V$ (strong holomorphy). 
Note first that $\partial_{z}u=\psi$. 
Now consider $p$. 
We have
$$
\frac{1}{h}\bigl(p(z+h)-p(z)\bigr)
=
\frac{1}{h}\Bigl(G\bigl(\ba+(z+h)\psi\bigr)-G\bigl(\ba+z\psi \bigr)\Bigr)=:r
\;.
$$
By Lemma \ref{lemmastab} we deduce that 
$$
\|r\|_{V} \le \frac{\|f\|_{H^{-1}(D)}}{\amin^2}\|\psi\|_{\Li}\;.
$$
From this it follows that there is a weakly convergent
subsequence in $V$, as $|h| \to 0.$ We proceed to
deduce existence of a strong limit.
%
%
To this end, we introduce the sesquilinear form
$$
b(p,q)=\int_{D}u\nabla p{\overline \nabla q} dx \;.
$$
Then
$$
b\bigl(G(u),q\bigr)=(f,q) \quad \forall q \in V\;.
$$
For a coefficient function $u$ as in \eqref{assume1},
the form $b(\cdot,\cdot)$ 
is equal to the parametric sesquilinear form
$
\alpha(z;p,q)
$
defined in \eqref{eq:alv1}.

Note that for $z=\bar{a} + y\psi\in \IR$ and for 
real-valued arguments $p$ and $q$, the 
parametric sesquilinear form $\alpha(z;p,q)$ 
coincides with the bilinear form in \eqref{weaksol}.
Accordingly, for every $z\in \IC^{\IJ}$ the unique 
holomorphic extension of the parametric solution
$G(u(\bar{a}+y\psi))$ to complex parameters 
$z=y+i\zeta$ is the unique variational solution
of the parametric problem
\begin{equation}
\label{eq:alv2}
\alpha\bigl(z;G(\ba+z\psi),q\bigr)=(f,q),\quad \forall q \in V.
\end{equation}
Assumption ${\bf UEAC}(\amin,\amax)$ is readily seen to 
imply
$$
\forall p\in V:
\quad
\re\bigl(\alpha(z;p,p)\bigr) 
\ge 
\amin\|p\|_V^2
\;.
$$
If we choose $\delta \in (\kappa,1)$ and choose 
$z=y+i\eta$, we obtain, for all $\va$ and for 
$|y| \le \delta/\kappa$ 
\begin{equation}
\label{eq:need}
\re\bigl(\alpha(z;p,p)\bigr) \ge \aminb(1-\delta)\|p\|_V^2.
\end{equation}

From \eqref{eq:alv2} we see that
for such values of $z=y+i\va$
\begin{eqnarray*}
0=&\alpha\Bigl(z;G\bigl(\ba+z\psi\bigr),q\Bigr)-\alpha\Bigl(z;G\bigl(\ba+(z+h)\psi\bigr),q\Bigr)\\
&+\alpha\Bigl(z;G\bigl(\ba+(z+h)\psi\bigr),q\Bigr)
-\alpha\Bigl(z+h;G\bigl(\ba+(z+h)\psi\bigr),q\Bigr)\\
=&\alpha\Bigl(z;G\bigl(\ba+z\psi\bigr)-G\bigl(\ba+(z+h)\psi\bigr),q\Bigr)\\
&-\int_{D} h\psi \nabla G\bigl(\ba+(z+h)\psi\bigr)
{\overline \nabla q}dx.
\end{eqnarray*}
Dividing by $h$ we obtain that $r$ satisfies,
for all $z=y+i\va$ with 
$|y|\leq \delta/\kappa$ and every $\va\in \mathbb{R}$
\begin{equation}
\label{eq:r}
\forall q\in V:
\qquad \alpha\bigl(z;r,q\bigr)+\int_{D} 
\psi \nabla G\bigl(\ba+(z+h)\psi\bigr) \overline{\nabla q}dx = 0 \;.
\end{equation}
The second term we denote by $s(h)$ and note that, by
Lemma \ref{lemmastab},
$$
|s(h_1)-s(h_2)| \le \frac{1}{\amin^2}\|\psi\|_{\infty}^2\|f\|_{1}\|q\|_{V}|h_1-h_2| \;.
$$ 
If we denote the solution $r$ to equation \eqref{eq:r} by
$r_h(\ba;z)$ then we deduce from the Lipschitz continuity
of $s(\cdot)$ that $r_h(\ba;z) \to r_0(\ba;z)$
where
$$\alpha(z;r_0,q)=s(0), \quad \forall q \in V.$$
Hence $r_0=\partial_{z} G(\ba+z\psi) \in V$ and
we deduce that $G:[-1,1] \to V$ can be extended
to a complex-valued function which is 
strongly holomorphic on the strip 
$\{y+i\va: |y|<\delta/\kappa, \; \va \in\mathbb{R}\}$.

We next study the domain of holomorphy of the 
analytic continuation of the potential 
$\Phi(\ba+z\psi;d)$ to parameters $z\in \IC$. 
It suffices to consider $K=1$
noting that then the unique analytic continuation
of the potential $\Phi$ is given by
\beq\label{eq:Phicontin}
\Phi(\ba+z\psi;\delta)=
\frac{1}{2\gamma^2}
\Bigl(\delta - \cG(\ba+z\psi)\Bigr)^\top
\Bigl(\delta - \cG(\ba+z\psi)\Bigr).
\eeq
The function $z \mapsto \cG(\ba+z\psi)$ is
holomorphic with the same domain of holomorphy as $G(\ba+z\psi)$.
Similarly it follows that the function
$$
z \mapsto \Bigl(\delta - \cG(\ba+z\psi)\Big)^\top
           \Bigl(\delta - \cG(\ba+z\psi)\Big)
$$
is holomorphic, with the same domain of holomorphy;
this shown by composing the relevant power series
expansion.
From this we deduce that $\Theta$ and $\Psi$ are holomorphic,
with the same domain of holomorphy.
\end{proof}

So far we have considered the case $\IJ=\{1\}$\;.
We now generalize. 
To this end, 
we pick an arbitrary $m \in \bbJ$ and write
$y=(\ys,y_m)$ and $z=(\zs,z_m)$ \;.

\begin{ass}
\label{a:2}
There are constants 
$0< \aminb \le \amaxb < \infty$ and $\kappa \in (0,1)$
such that
\beq
0<\aminb \le \ba \le \amaxb <\infty,\quad \hbox{a.e.}\,\, x \in D,
\quad
\Big\|\|\psi_j\|_{L^{\infty}(D)}\Bigr\|_{\ell^1(\bbJ)} < \kappa\aminb\;.
\eeq
\end{ass}

For $m\in\mathbb{J}$, we write \eqref{assume1} in the form
$$
u(x;y)=\ba(x)+y_m\psi_m(x)+\sum_{j \in \IJ\backslash\{m\}} y_j\psi_j(x)\;.
$$
From Assumption \ref{a:2} we deduce that there
are numbers $\kj \le \kappa$ such that
$$\|\psi_j\|_{L^{\infty}}<\aminb \kj.$$
Hence we obtain, for every $x\in D$ and every $y\in U$ 
the lower bound
\begin{eqnarray*}
u(x,y) & \ge \aminb\Bigl(1-\bigl(\kappa-\km\bigr)-\km\Bigr)
\\
& 
\ge 
\aminb
\Bigl(
1-\bigl(\kappa-\km\bigr)\Bigr)\Bigl(1-\frac{\km}{1-\bigl(\kappa-\km\bigr)}
\Bigr)
\\
& \ge \amp(1-\kmp)
\end{eqnarray*}
with 
$\amp=\amin(1-\kappa)$ 
and 
$\kmp=\km\Bigl(1-\bigl(\kappa-\km\bigr)\Bigr)^{-1} \in (0,1)\;.$
With this observation we obtain
\begin{lemma} \label{l:two}
Let Assumption \ref{a:2} hold and set $U=[-1,1]^{\bbJ}$
and $\phi=G:U \to V$.
Then the functions
$\Psi:U \to V$ and $\Theta:U \to \bbR,$ 
as well as the potential 
$\Phi(u(x,\cdot);\delta):U \to \bbR$
admit unique extensions to strongly holomorphic functions
on the product of strips given by
\beq\label{eq:Sprod}
\cS_\rho :=
\bigotimes_{j \in \bbJ}
\big\{
y_j+iz_j: |y_j|<\delta_j/\kjp,\quad z_j \in \mathbb{R} 
\big\}
\eeq
for any sequence $\rho = (\rho_j)_{j\in \IJ}$
with  $\rho_j \in (\kjp,1)$.
\end{lemma}

\begin{proof} 
Fixing $\ys$, 
we view $\Psi$ and $\Theta$ as functions
of the single parameter $y_m$.
For each fixed $\ys$, 
we extend $y_m$ to a complex variable $z_m$. 
The estimates preceding the
statement of this lemma, together with Lemma \ref{l:one},
show that $\Psi$ and $\Theta$ are holomorphic in the
strip $\{y_m+iz_m: |y_m|<\delta_m/\kmp\}$ for 
any $\delta_m \in (\kmp,1)$. 
Hartogs' theorem \cite{HoermandComplex}
and the fact that in
separable Banach spaces (such as $V$) weak 
holomorphy equals strong holomorphy
extends this result onto the product of strips,
$\cS$.
\end{proof}
We note that the strip $\cS_\rho\subset\IC^\IJ$ 
defined in \eqref{eq:Sprod} contains in particular 
the polydisc
$\cU_{\rho}$ with $(\rho_j)_{\j\in \IJ}$ 
where 
$\rho_j = \delta_j/\kjp$.

\subsection{Holomorphy and bounds on the posterior density}
\label{ssec:HolPost}
So far, we have shown that the responses $G(u)$, $\cG(u)$
and the potentials $\Phi(u;\delta)$ depend holomorphically 
on the coordinates $z\in \cA_r \subset \IC^{\IJ}$ in the 
parametric representation 
$u = \bar{a} + \sum_{j\in\IJ} z_j\psi_j$.
Now we deduce bounds on the analytic continuation 
of the posterior density $\Theta(z)$ in \eqref{eq:PostDens}
as a function of the parameters $z$ on the domains 
of holomorphy. We have
\begin{theorem}\label{thm:PhiArbound}
Under ${\bf UEAC}(\amin,\amax)$
for the analytic continuation $\Theta(z)$ of the posterior density
to the domains  $\cA_r$ of holomorphy defined in \eqref{eq:Adelta},
i.e. for 
\beq \label{eq:PostDensz}
\Theta(z) 
= 
\exp\left(-\Phi(u;\delta)|_{u = \bar{a} + \sum_{j\in\IJ} z_j\psi_j}\right)
\eeq
there holds for every $0 < r < \amin$
\beq \label{eq:PhiArbound}
\sup_{z\in \cA_r} | \Theta(z) |
=
\sup_{z\in \cA_r} | \exp(-\Phi(u(z);\delta) | 
\leq 
\exp
\left( 
\frac{\| f \|_{V^*}^2}{r^2} \sum_{k=1}^K \| o_k \|_{V*}^2 
\right).
\eeq
\end{theorem}
These analyticity properties, and resulting bounds, can
be extended to functions $\phi(\cdot)$ as defined by
\eqref{eq:defmpointcorp}, using 
Lemma \ref{l:two} and Theorem \ref{thm:PhiArbound}. 
This gives the following result.
\begin{corollary} \label{cor:PsiAnalyt}
Under ${\bf UEAC}(\amin,\amax)$, 
for any $m\in \mathbb{N}$ the functionals 
$\phi(u) = p^{(m)} \in S = V^{(m)}$
the posterior densities 
$\Psi(z) = \Theta(z) \phi(u(z))$ 
defined in \eqref{eq:psi}
admit analytic continuations as strongly holomorphic, 
$V^{(m)}$-valued functions with 
domains $\cA_r$ of holomorphy defined in \eqref{eq:Adelta}.
Moreoever, for these functionals the 
analytic continuations of $\Psi$ in \eqref{eq:psi} admit 
the bounds
\beq\label{eq:Psiboundpm}
\sup_{z\in\cA_r} \| \Theta(z) (p(z))^{(m)} \|_{V^{(m)}}
\leq
\frac{\| f \|_{V^*}^m}{r^m}
\exp
\left(
\frac{\| f \|_{V^*}^2}{r^2} \sum_{k=1}^K \| o_k \|_{V*}^2
\right)
\;.
\eeq
\end{corollary}
\section{Polynomial Chaos Approximations of the Posterior}
\label{sec:PCApprTh}
Building on the results of the previous section, we 
now proceed to approximate $\Theta(z),$
viewed as a holomorphic functional 
over $z \in \bbC^{\bbJ}$, by so-called 
{\em polynomial chaos representations}.
Exactly the same results on analyticity
and on $N$-term approximation of $\Psi(z)$
hold. We omit details for reasons of brevity
of exposition and confine ourselves
to establishing rates of convergence
of $N$-term truncated representations of the posterior
density $\Th$. 
The results in the present section are, in one sense,
{\em sparsity results} on the posterior density $\Th$.
On the other hand, 
such $N$-term truncated gpc representations of $\Th$ are, 
as we will show in the next section, computationally accessible
once sparse truncated adaptive forward solvers of the parametrized 
system of interest are available. Such solvers are
indeed available 
(see, e.g., \cite{BAS09,CCDS11,CS+CJG11} and the references therein), 
so that the abstract approximation results in the present section
have a substantive constructive aspect. 
Algorithms based on Smolyak-type quadratures in $U$ 
which are designed based on the present theoretical results 
will be developed and analyzed in \cite{AScSt}.
In this section we analyze the convergence rate
of  {\em $N$-term truncated Legendre gpc-approximations of 
$\Th$} and, 
with the aim of a {\em constructive $N$-term approximation}
of the posterior $\Th(y)$ in $U$ in Section \ref{sec:Expost} 
ahead, we analyze also $N$-term truncated 
{\em monomial gpc-approximations} of $\Th(y)$.

\subsection{gpc Representations of $\Th$}
%
With the index set $\IJ$ from the parametrization \eqref{assume1} 
of the input, we associate 
the countable index set 
\beq
\label{eq:defcF}
\cF=\{\nu \in \bbN_{0}^{\bbJ}:|\nu|_{1}<\infty\}
\eeq
of multiindices where $\bbN_{0}=\bbN \cup \{0\}$.
We remark that sequences $\nu\in\cF$ are finitely supported 
even for $\bbJ=\bbN$.
For $\nu\in\cF$, we denote by
$\bbI_{\nu} =\{ j\in \IN: \nu_j \ne 0\} \subset \IN$ 
the ``support'' of $\nu\in \cF$, i.e.
the finite set of indices of entries of $\nu\in\cF$ 
which are non-zero, and by 
$\aleph(\nu)  := \# \bbI_\nu < \infty$, $\nu\in\cF$
the ``support size'' of $\nu$, i.e. the
cardinality of $\bbI_\nu$.

For the deterministic approximation of the posterior density 
$\Theta(y)$ in \eqref{eq:PostDens} we shall use tensorized
polynomial bases similar to what is done in so-called
``polynomial chaos'' expansions of random fields.
We shall consider two particular polynomial bases, Legendre
and monomial bases.
\subsubsection{Legendre Expansions of $\Th$}
\label{sssec:LegExpTh}
Since we assumed that the prior measure
$\mu_0(dy)$ is built by tensorization of the uniform 
probability measures on $(-1,1)$, we build the bases by 
tensorization as follows:
let $L_{k}(z_j)$ denote the $k^{th}$ Legendre polynomial
of the variable $z_j \in \IC$, normalized such that 
\beq\label{eq:LegNorm}
\int_{-1}^1 (L_k(t))^2 \frac{dt}{2} = 1,\quad k = 0,1,2,...
\eeq
Note that $L_0 \equiv 1$. 
The Legendre polynomials $L_k$ in \eqref{eq:LegNorm}
are extended to tensorproduct polynomials on $U$ via
\beq\label{eq:MultiLeg}
L_{\nu}(z)=\prod_{j \in \IJ} L_{\nu_j}(z_j),
\quad
z\in\IC^\IJ,\; \nu\in \cF \;.
\eeq
The normalization \eqref{eq:LegNorm} implies
that the polynomials $L_\nu(z)$ in \eqref{eq:MultiLeg}
are well-defined for any $z\in \IC^\IJ$
since the finite support of each element of $\nu \in \cF$ 
implies that $L_{\nu}$ in \eqref{eq:MultiLeg} 
is the product of only finitely many nontrivial polynomials.
It moreover implies that the 
set of tensorized Legendre polynomials 
\beq \label{eq:poly}
\IL(U,\mu_0(dy)) 
:= 
\{ L_\nu: \nu \in \cF \} 
\eeq
forms a countable orthonormal basis in $L^2(U,\mu_0(dy))$.
This observation suggests, by virtue of
Lemma \ref{l:L2} below, the use of mean square
convergent gpc-expansions
to represent  $\Th$ and $\Psi$. 
Such expansions can also serve as a basis
for sampling of these quantities with
draws that are equidistributed with respect
to the prior $\mu_0$.
\begin{lemma} \label{l:L2} 
The density $\Th: U \to \IR$ is square 
integrable with respect to the prior 
$\mu_0(dy)$ over $U$, 
i.e. $\Th\in L^2\bigl(U,\mu_0(dy)\bigr)$.
Moreover, if the functional $\phi(\cdot):U\rightarrow S$ 
in \eqref{eq:psi} is bounded, then 
$$ \int_{U} \|\Psi(y)\|_{S}^2 \mu_0(dy)<\infty,$$
i.e. $\Psi\in L^2\bigl(U,\mu_0(dy);S\bigr).$ 
\end{lemma}

\begin{proof} 
Since $\Phi$ is positive it follows
that $\Theta(y) \in [0,1]$ for all $y \in U$
and the first result follows
because $\mu_0$ is a probability measure. 
Now define $K=\sup_{y\in U}|\phi(y)|.$ 
Then $\sup_{y\in U}\|\Psi(y)\|_{S} \le K$ and
the second result follows similarly, again
using that $\mu_0$ is a probability measure.
\end{proof}

\begin{remark}
It is a consequence of \eqref{cLM2} that 
in the case where $\phi(u)=G(u)=p \in V$ we have
$\|\Psi(y)\|_{V} \le \|f\|_{V^*}/\amin$ 
for all $y \in U$. 
Thus the second assertion of Lemma \ref{l:L2}
holds for calculation of the expectation of
the pressure under the posterior distribution on $u$.
Indeed the assertion holds for all moments of the
pressure, the concrete examples which we concentrate on here.
\qed
\end{remark}
Since  $\IL(U,\mu_0(dy))$ in \eqref{eq:poly}
is a countable orthonormal basis
of $L^2(U,\mu_0(dy))$, the density $\Th(y)$ 
of the posterior measure given data $\delta\in Y$,
and the posterior reweighted pressure $\Psi(y)$
can be represented in $L^2(U,\mu_0(dy))$ 
by (parametric and deterministic) 
generalized Legendre polynomial chaos expansions.
We start by considering the scalar valued function
$\Theta(y)$.
\begin{equation} \label{eq:exp}
\Theta(y)=\sum_{\nu \in \cF}\theta_{\nu}L_{\nu}(y)
\quad \mbox{in} \quad L^2(U,\rho(dy))
\end{equation}
where the gpc expansion coefficients $\theta_\nu$ 
are defined by
\beq\label{unu}
\theta_{\nu}
=
\int_{U} \Th(y)L_{\nu}(y)\mu_0(dy) \;, \quad \nu\in \cF \;.
\eeq
By Parseval's equation and the normalization \eqref{eq:LegNorm}, 
it follows immediately from \eqref{eq:exp} and Lemma \ref{l:L2}
with Parseval's equality
that the second moment of the posterior density with respect to
the prior
\beq\label{eq:Parsev}
\| \Theta \|_{L^2(U,\mu_0(dy))}^2
=
\sum_{\nu\in \cF} |\theta_\nu|^2
\eeq
is finite.
\subsubsection{Monomial Expansions of $\Th$}
\label{sssec:MonExpTh}
We next consider expansions of the posterior density 
$\Th$ with respect to monomials 
$$
y^\nu = \prod_{j\geq 1} y_j^{\nu_j},\quad y\in U,\quad \nu\in \cF\;.
$$
Once more, the infinite product is well-defined since,
for every $\nu\in \cF$, it contains only $\aleph(\nu)$ many nontrivial factors.
By Lemma \ref{l:two} and Theorem \ref{thm:PhiArbound}, 
the posterior density $\Th(y)$ admits an analytic continuation 
to the product of strips $\cS_\rho$ which contains, in particular,
the polydisc $\cU_\rho$. In $U$, $\Th(y)$ can therefore be represented
by a  monomial expansion with uniquely determined 
coefficients $\tau_\nu\in V$ which coincide,
by uniqueness of the analytic continuation,
with the Taylor coefficients of $\Th$ at $0\in U$:
\beq\label{eq:TayTh}
\forall y\in U:\quad
\Th(y) = \sum_{\nu\in \cF} \tau_\nu y^\nu \;,
\quad
\tau_\nu := \frac{1}{\nu!} \partial^\nu_y \Th(y) \mid_{y=0}
\;.
\eeq
\subsection{Best $N$-term Approximations of $\Th$}
\label{ssec:BestNTh}
In our deterministic parametric approach to
Bayesian estimation, evaluation of expectations 
under the posterior
requires evaluation of the integrals \eqref{eq:Z}
and \eqref{eq:intpsi}. 
Our strategy is
to approximate these integrals 
by truncating the spectral
respresentation \eqref{eq:exp}, as well
as a similar expression for $\Psi(y)$, to a 
finite number $N$ of significant terms,
and to estimate the error incurred by doing so.
It is instructive to compare with Monte Carlo
methods.  Under the conditions of Lemma \ref{l:L2},
posterior expectation of functions $\Psi$
have finite second moments
so that Monte Carlo methods exhibit
the convergence rate $N^{-1/2}$ 
in terms of the number $N$ of samples, with
similar extension to MCMC methods. 
Here, however, we will show that it is possible to
derive approximations which incur error decaying
more quickly that the square root of $N$, 
where $N$ is now the
number of significant terms retained in \eqref{eq:exp}.

By \eqref{eq:Parsev}, the coefficient sequence
$(\theta_\nu)_{\nu\in\cF}$ must necessarily
decay. If this decay is sufficiently strong,
possibly high convergence rates of 
$N$-term approximations of the integrals
\eqref{eq:Z}, \eqref{eq:intpsi} occur.
The following classical result from 
approximation theory \cite{dev98}
makes these heuristic 
considerations precise:
denote by $(\gamma_n)_{n\in \IN}$ a 
(generally not unique) decreasing rearrangement 
of the sequence $(|\theta_\nu|)_{\nu\in \cF}$.
Then,
for any summability exponents 
$0 < \sigma \leq q \leq \infty$ and 
for any $N\in \mathbb{N}$ holds
\beq
\left(\sum_{n>N} \gamma_n^q\right)^{\frac 1 q}
\leq
N^{-(\frac1 \sigma - \frac 1 q) }
\left(\sum_{n\geq 1} \gamma_n^\sigma \right)^{\frac 1 \sigma}
\;.
\label{stechkin}
\eeq
\subsubsection{$L^2(U;\mu_0)$ Approximation.}
\label{sssec:L2ApprTh}
Denote by $\Lambda_N \subset \cF$ a set of indices $\nu\in \cF$
corresponding to $N$ largest gpc coefficients $|\theta_\nu|$
in \eqref{eq:exp}, and denote by
\beq\label{eq:SLambN}
\Th_{\Lambda_N}(y) := \sum_{\nu\in\Lambda_N} \theta_\nu L_\nu(y)
\eeq
the Legendre expansion \eqref{eq:exp} truncated to this set of
indices.
Using \eqref{stechkin} with $q=2$, 
Paseval's equation \eqref{eq:Parsev}  
and $0 < \sigma \le 1$ we obtain
for all $N$
\beq\label{eq:ThNRate}
\| \Th(z) - \Th_{\Lambda_N}(z) \|_{L^2(U,\mu_0(dy))}
\leq 
N^{-s}\| ( \theta_\nu )\|_{\ell^\sigma(\cF)},\;\; s:=\frac 1 \sigma - \frac 1 2 \;.
\eeq
We infer from \eqref{eq:ThNRate} that
{\em 
a mean-square convergence rate $s>1/2$ 
of the approximate posterior density $\Theta_{\Lambda_N}$
can be achieved {\em provided that} 
$(\theta_\nu)\in \ell^\sigma(\cF)$ 
for some $0 < \sigma < 1$}. 
\subsubsection{$L^1(U;\mu_0)$ and pointwise Approximation of $\Th$}
\label{sssec:L1ApprTh}
The analyticity of $\Th(y)$ in $\cU_\rho$ implies
that $\Th(y)$ can be represented by the Taylor 
exansion \eqref{eq:TayTh}. 
This expansion is unconditionally summable in $U$
and, for any sequence $\{\Lambda_N\}_{N\in \IN}\subset\cF$
which exhausts $\cF$
\footnote{
We recall that a sequence 
$\{\Lambda_N\}_{N\in \IN}\subset\cF$
of index sets $\Lambda_N$ whose cardinality 
does not exceed $N$ exhausts $\cF$
if any finite $\Lambda\subset \cF$ is
contained in all $\Lambda_N$ for $N\geq N_0$ with $N_0$
sufficiently large.}, 
the corresponding sequence of 
$N$-term truncated partial Taylor sums 
\beq\label{eq:TLamN}
T_{\Lambda_N}(y) := \sum_{\nu\in \Lambda_N} \tau_\nu y^\nu
\eeq
converges pointwise in $U$ to $\Th$.
Since for $y\in U$ and $\nu\in \cF$ we have $|y^\nu|\leq 1$,
for any $\Lambda_N\subset \cF$ 
of cardinality not exceeding $N$ holds
\beq\label{eq:LooAppTh}
\sup_{y\in U}
\left| \Theta(y) - T_{\Lambda_N}(y) \right|
=
\sup_{y\in U}
\left|
\sum_{\nu\in \cF\backslash \Lambda_N}
\tau_\nu y^\nu
\right|
\leq
\sum_{\nu\in \cF\backslash \Lambda_N} 
|\tau_\nu|
\;.
\eeq
Similarly, we have 
$$
\left\| \Th - T_{\Lambda_N} \right\|_{L^1(U,\mu_0)}
=
\left\| 
\sum_{\nu\in\cF\backslash\Lambda_N} \tau_\nu y^\nu 
\right\|_{L^1(U,\mu_0)}
\leq
\sum_{\nu\in\cF\backslash\Lambda_N}
|\tau_\nu| 
\left\| y^\nu \right\|_{L^1(U,\mu_0)}\;.
$$
For $\nu\in \cF$, we calculate
$$
\begin{array}{rcl}
\left\| y^\nu \right\|_{L^1(U,\mu_0)}
&=& \ds
\int_{y\in U} |y^\nu| \mu_0(dy)
= 
\frac{1}{(\nu+\b1)!}
\end{array}
$$
so that we find
\beq\label{eq:L1AppTh}
\left\| \Th - T_{\Lambda_N} \right\|_{L^1(U,\mu_0)}
\leq
\sum_{\nu\in\cF\backslash\Lambda_N}
\frac{
|\tau_\nu| 
}{
(\nu+\b1)!}
\;.
\eeq

\subsubsection{Summary}

There are, hence, two main issues to be addressed to
employ the preceding approximations in practice: 
i) establishing the summability of the coefficient
   sequences in the series \eqref{eq:exp}, \eqref{eq:TayTh}; 
and
ii) finding algorithms which locate sets $\Lambda_N\subset\cF$ 
of cardinality not exceeding $N$ for which the truncated
partial sums preserve the optimal convergence rates
and, once these sets are localized,
to determine the $N$ ``active'' coefficients $\theta_\nu$
or $\tau_\nu$,
preferably in close to $O(N)$ operations.
In the remainder of this section, we address i) and consider ii) 
in the next section.

\subsection{Sparsity of the posterior density $\Theta$}
\label{ssec:SpTh}
The analysis in the previous section shows that
the convergence rate of the truncated gpc-type approximations
\eqref{eq:SLambN}, \eqref{eq:TLamN}  
on the parameter space $U$ is determined 
by the $\sigma$-summability of the
corresponding coefficient sequences
$(|\theta_\nu|)_{\nu\in \cF}$, $(|\tau_\nu|)_{\nu\in \cF}$\;.
We now show that summability (and, hence, sparsity) 
of Legendre and Taylor coefficient sequences 
in the expansions \eqref{eq:exp}, \eqref{eq:TayTh}
is determined by that of the sequence 
$(\|\psi_j\|_{L^\infty(D)})_{j\in\mathbb{N}}$ 
in the input's fluctuation expansion \eqref{assume1}. 
Throughout, Assumptions \ref{assump1} and \ref{assump2} 
will be required to hold.
We formalize the decay of the $\psi_j$ in \eqref{eq:par} by
\begin{ass}\label{summabilityofpsi}
There exists $0<\sigma<1$ such that for
the parametric representations \eqref{assume1},
\eqref{eq:par} it holds that
\beq\label{asspsi}
\sum_{j=1}^\infty \|\psi_j\|_{L^\infty(D)}^\sigma < \infty \;.
\eeq
\end{ass}
The strategy of establishing sparsity of the sequences 
$(|\theta_\nu|)_{\nu\in \cF}$, $(|\tau_\nu|)_{\nu\in \cF}$ is based
on estimating the sequences by Cauchy's integral formula 
applied to the analytic continuation of $\Theta$.
\subsubsection{Complex extension of the parametric problem}\label{estimateforunu}
To estimate $|\theta_\nu|$ in \eqref{eq:SLambN}, we shall use 
the holomorphy of solution to the 
(analytic continuation of the) parametric deterministic problem:
let $0<K<1$ be a constant such that
\beq\label{eq:defK}
K\sum_{j=1}^\infty\|\psi_j\|_{L^\infty(D)}<{a_{\min}\over 8}.
\eeq
Such a constant exists by Assumption \ref{summabilityofpsi}.  
For $K$ selected in this fashion, 
we next choose an integer $J_0$ such that
\[
\sum_{j> J_0}\|\psi_j\|_{L^\infty(D)}<{a_{\min}K\over 24(1+K)}.
\]
Let $E=\{1,2,\ldots,J_0\}$ and $F=\IN\setminus E$. 
We define
\[
|\nu_F|=\sum_{j> J_0}|\nu_j|.
\]
For each $\nu\in {\cF}$ we define a $\nu$-dependent
radius vector $\xr = (r_m)_{m\in \IJ}$ with 
$r_m > 0$ for all $m\in \IJ$ as follows:
\beq\label{eq:defrm}
r_m = K\ \mbox{when}\ m\le J_0\ \mbox{and}\ 
r_m = 1 + {a_{\min}\nu_m\over 4|\nu_F|\|\psi_m\|_{L^\infty(D)}}\ \mbox{when}\ m>J_0,
\eeq
where we make the convention that 
${|\nu_j|\over |\nu_F|}=0$ if $|\nu_F|=0$. 
We consider the open discs ${\cU}_m\subset \mathbb{C}$ 
defined by
\beq\label{eq:defUm}
[-1,1] \subset {\cU}_m :=\{z_m\in{\IC}: |z_m| < 1+r_m\} \subset\IC.
\eeq
We will extend the parametric deterministic problem \eqref{eq:sesqfwdpb}
to parameter vectors $z$ in the polydiscs
\beq\label{eq:defUm'}
{\cU}_{1+ \xr} := \bigotimes_{m\in \IJ}{\cU}_m \subset {\mathbb C}^{\mathbb J}.
\eeq
To do so,
we invoke the analytic continuation of the 
parametric, deterministic coefficient function $u(x,y)$ in 
(\ref{assume1}) to $z\in{\cU}$ which is for such $z$
{\em formally} given by
\[
u(x,z)=\bar a(x)+\sum_{m\in \IJ} \psi_m(x)z_m.
\]
We verify that
this expression is meaningful for $z\in{\cU}_{\xr}$:
we have, for almost every $x\in D$, 
\beas
|u(x,z)|&\le&\bar a(x)+\sum_{m\in \IJ}|\psi_m(x)|(1+r_m)
\\
&\le& \ds
{\rm ess}\sup_{x\in D} |\bar a(x)|
+
\sum_{m=1}^{J_0}\|\psi_m\|_{L^\infty(D)}(1+K)
\\
& & \ds
+
\sum_{m > J_0}
\Bigl(
2+{a_{\min}\nu_m\over 4|\nu_F|\|\psi_m\|_{L^\infty(D)}}
\Bigr)
\|\psi_m\|_{L^\infty(D)}
\\
&\le&\ds
\| \bar a \|_{L^\infty(D)} 
+ 
2\sum_{m=1}^\infty\|\psi_m\|_{L^\infty(D)}+{a_{\min}\over 4}
\;.
\eeas
\subsubsection{Estimates of the $\theta_\nu$}
\begin{proposition}\label{normofthnu}
There exists a constant $C>0$ such that, 
with the constant $K\in (0,1)$ in \eqref{eq:defK},
for every $\nu \in \cF$ the following estimate holds
\beq
|\theta_\nu| \le C\biggl(\prod_{m\in \bbI(\nu)}{2(1+K)\over K}\eta_m^{-\nu_m}\biggr),
\label{unux}
\eeq
where $\eta_m:=r_m+\sqrt{1+r_m^2}$ with $r_m$ 
as in (\ref{eq:defrm}).
\end{proposition}
{\it Proof}\ \ 
For $\nu\in \cF$, define $\theta_\nu$ by \eqref{unu}
let $S=\bbI(\nu)$ and define $\bar S=\IJ\setminus S$. 
For $S$ denote by ${\cU}_S=\otimes_{m\in S}{\cU}_m$ 
and 
${\cU}_{\bar S}=\otimes_{m\in \bar S}{\cU}_m$, 
and by $y_S=\{y_i: i\in S\}$ the extraction from $y$. 
Let ${\cE}_m$ be the ellipse in ${\cU}_m$ with foci at 
$\pm 1$ and semiaxis sum $\eta_m > 1$. 
Denote also ${\cE}_S = \prod_{m\in\bbI(\nu)}{\cE}_m$. 
We can then write (\ref{unu}) as
\[
\theta_\nu
=
{1\over (2\pi i)^{|\nu|_0}}
\int_{U}L_\nu(y)\oint_{{\cE}_S}{\Theta(z_S,y_{\bar S})\over (z_S-y_S)^{\bf 1}}
dz_Sd\rho(y).
\]
For each $m\in \IN$, let $\Gamma_m$ be a copy of $[-1,1]$ and $y_m\in \Gamma_m$. 
We denote by 
$U_S=\prod_{m\in S}\Gamma_m$ 
and 
$U_{\bar S}=\prod_{m\in\bar S}\Gamma_m$. 
We then have
\[
\theta_\nu
=
{1\over (2\pi i)^{|\nu|_0}}
\int_{U_{\bar S}}\oint_{{\cE}_{S}}
\Theta(z_S,y_{\bar S})\int_{U_S}{L_\nu(y)\over (z_S-y_S)^{\bf 1}}
d\rho_S(y_S)dz_Sd\rho_{\bar S}(y_{\bar S}).
\]
To proceed further, we recall the definitions of 
the Legendre functions of the second kind
\[
Q_n(z)=\int_{[-1,1]}{L_n(y)\over(z-y)}d\rho(y).
\]
Let $\nu_S$ be the restriction of $\nu$ to $S$. 
We define
\[
{\cQ}_{\nu_S}(z_S)=\prod_{m\in\bbI(\nu)}Q_{\nu_m}(z_m).
\]
Under the Joukovski transformation 
$z_m={1\over 2}(w_m+w_m^{-1})$, 
the Legendre polynomials of the second kind 
take the form
\[
Q_{\nu_m}({1\over 2}(w_m+w_m^{-1}))=\sum_{k=\nu_m+1}^\infty{q_{{\nu_m}k}\over w_m^k}
\]
with $|q_{{\nu_m}k}|\le \pi$. 
Therefore
\[
|{\cQ}_{\nu_S}(z_S)|
\le 
\prod_{m\in S}\sum_{k=\nu_m+1}^\infty{\pi\over\eta_m^k}
=
\prod_{m\in S}\pi{\eta_m^{-\nu_m-1}\over 1-\eta_m^{-1}}.
\]
We then have
\beas
|\theta_\nu|
&=&
\Bigg|{1\over (2\pi i)^{|\nu|_0}}
\int_{U_{\bar S}}\oint_{{\cE}_S}
\Theta(z_S,y_{\bar S}){\cQ}_{\nu_S}(z_S)dz_Sd\rho_{\bar S}(y_S)
\Bigg|
\\
&\le& \ds
{1\over (2\pi)^{|\nu|_0}}
\int_{U_{\bar S}}\oint_{{\cE}_S}
|\Theta(z_S,y_{\bar S})|{\cQ}_{\nu_S}(z_S)dz_Sd\rho_{\bar S}(y_S)
\\
&\le&\ds
{1\over (2\pi)^{|\nu|_0}}
\|\Theta(z)\|_{L^\infty({\cE}_S\times U_{\bar S})}
\max_{{\cE}_S}|{\cQ}_{\nu_S}|
\prod_{m\in S}{\rm Len}({\cE}_m)
\\
&\le&
{1\over (2\pi)^{|\nu|_0}}\|\Theta(z)\|_{L^\infty({\cE}_S\times{\cU}_{\bar S})}
\prod_{m\in S}\pi{\eta_m^{-\nu_m-1}\over 1-\eta^{-1}_m}{\rm Len}({\cE}_m)
\\
&\le&
C\prod_{m\in S}{2(1+K)\over K}\eta_m^{-\nu_m},
\eeas
as ${\rm Len}({\cE}_m)\le 4\eta_m$, 
$\eta_m\ge 1+K$ and as $|\Th(z)|$ is uniformly bounded
on ${\cE}_S\times{\cU}_{\bar S}$ by Theorem \ref{thm:PhiArbound}.
\hfill$\Box$
\subsubsection{Summability of the $\theta_\nu$}
\label{sssec:Sumthnu}
To show the ${\ell}^\sigma ({\cF})$ summability of $|\theta_\nu|$, 
we use the following result, which appears as Theorem 7.2 in \cite{CDS1}.
\begin{proposition}\label{psummability} 
For 
$0 < \sigma < 1$ and for 
any sequence $(b_\nu)_{\nu\in \cF}$,
$$
\ds{\Bigl({|\nu|!\over\nu!}b^\nu\Bigr)_{\nu\in{\cF}}\in {\ell}^\sigma({\cF})}
\Longleftrightarrow
\sum_{m\ge 1} |b_m| < 1
\quad\mbox{and}\quad
(b_m)_{m\in \IN} \in {\ell}^\sigma (\IN)
\;.
$$
\end{proposition}
This result implies the $\sigma$-summability of the sequence
$(\theta_\nu)$ of Legendre coefficients.
\begin{proposition}\label{lp}
Under Assumptions \ref{assump1}, \ref{assump2},
for $0 < \sigma < 1$ as in Assumption \ref{summabilityofpsi},
$
\sum_{\nu\in{\cF}} |\theta_\nu|^\sigma
$
is finite.
\end{proposition}
{\it Proof}\ \ 
We have from Proposition \ref{normofthnu}
that
\beas
|\theta_\nu| & \le & C\prod_{m\in S}{2(1+K)\over K}(1+r_m)^{-\nu_m}
\\
&\le&
C\Bigl(\prod_{m\in E,\nu_m\ne 0}{2(1+K)\over K}\eta^{\nu_m}\Bigr)
\Bigl(\prod_{m\in F,\nu_m\ne 0}{2(1+K)\over K}
\Bigl({4|\nu_F|\|\psi_m\|_{L^\infty(D)}\over a_{\min}\nu_m}\Bigr)^{\nu_m}\Bigr)
\eeas
where $\eta=1/(1+K)<1$\;. 
Let ${\cF}_E = \{\nu\in{\cF}:\quad \bbI(\nu)\subset E\}$ 
and 
${\cF}_F={\cF}\setminus E$. 
From this, we have
\[
\sum_{\nu\in {\cF}} |\theta_\nu|^\sigma \le C A_E A_F
\]
where
\[
A_E = \sum_{\nu\in{\cF}_E} \prod_{m\in E,\nu_m\ne 0}
\Bigl({2(1+K)\over K}\Bigr)^\sigma \eta^{\sigma\nu_m},
\]
and
\[
A_F=\sum_{\nu\in{\cF}_F} 
\prod_{m\in F,\nu_m\ne 0}
\Bigl({2(1+K)\over K}\Bigr)^\sigma
\Bigl({4|\nu|\|\psi_m\|_{L^\infty(D)}\over a_{\min}\nu_m}\Bigr)^{\sigma \nu_m}.
\]
We estimate $A_E$ and $A_F$: 
for $A_E$, we have
\[
A_E=\biggl(1+\Bigl({2(1+K)\over K}\Bigr)^\sigma \sum_{m\ge 1}\eta^{pm}\biggr)^{J_0},
\]
which is finite due to $\eta<1$. 
For $A_F$, we note that for $\nu_m\ne 0$,
\[
{ 2(1+K)\over K} \le \Bigl({2(1+K)\over K}\Bigr)^{\nu_m}.
\]
Therefore
\[
A_F
\le 
\sum_{\nu\in{\cF}_F}
\prod_{m\in F}\Bigl({|\nu|d_m\over \nu_m}\Bigr)^{\sigma \nu_m},
\]
where
\[
d_m={8(1+K)\|\psi_m\|_{L^\infty(D)}\over Ka_{\min}} \;.
\]
With the convention that $0^0=1$ we obtain 
from the Stirling estimate
\[
{n!e^n\over e\sqrt{n}}\le n^n\le {n!e^n\over\sqrt{2\pi n}}
\]
that $ |\nu|^{|\nu|}\le |\nu|!e^{|\nu|}$.
Inserting this in the above bound for $A_F$, we obtain
\[
\prod_{m\in F}\nu_m^{\nu_m}
\ge 
{\nu!e^{|\nu|}\over\prod_{m\in F}\max\{1,e\sqrt{\nu_m}\}} \; .
\]
Hence
\[
A_F
\le 
\sum_{\nu\in{\cF}_F}
\Bigl({|\nu|!\over\nu!}d^{\nu}\Bigr)^\sigma
\Bigl(\prod_{m\in F}\max\{1,e\sqrt{\nu_m}\}\bigr)^\sigma
\le 
\sum_{\nu\in{\cF}_F}
\Bigl({|\nu|!\over\nu!}{\bar d}^{\nu}\Bigr)^\sigma
\;,
\]
where $ {\bar d}_m=ed_m $ and where we used the estimate 
$e\sqrt{n}\le e^n$.  
From this, we have
\[
\sum_{m\ge 1}\bar d_m
\le 
\sum_{m\in F}{24(1+K)\|\psi_m\|_{L^\infty(D)}\over Ka_{\min}}
\le 1.
\]
Since also
\[
\|\bar d\|_{l^\sigma(\IN)}<\infty
\]
we obtain with Proposition \ref{psummability}
the conclusion.\hfill$\Box$

We now show $\sigma$-summability of the Taylor coefficients
$\tau_\nu$ in \eqref{eq:TayTh}. To this end, we proceed as in 
the Legendre case: first we establish sharp bounds on the
$\tau_\nu$ by complex variable methods, and then show 
$\sigma$-summability of $(\tau_\nu)_{\nu\in \cF}$ 
by a sequence factorization argument.
\subsubsection{Bounds on the Taylor coefficients $\tau_\nu$}
\label{sssec:Boundtaunu}
\begin{lemma}\label{lem:taunubound}
Assume ${\bf UEAC}(\amin,\amax)$ and that 
$\rho = (\rho_j)_{j\geq 1}$ is an $r$-admissible
sequence of disc radii for some $0<r<\amin$.
Then the Taylor coefficients $\tau_\nu$ of 
the parametric posterior density \eqref{eq:TayTh}
satisfy 
\beq\label{eq:taunubound}
\forall \nu \in \cF:
\quad
| \tau_\nu | 
\leq 
\exp
\left(
\frac{\| f \|_{V^*}^2}{r^2} 
\sum_{k=1}^K \| o_k \|_{V*}^2
\right)
\prod_{j\geq 1} \rho_j^{-\nu_j}
\;.
\eeq
\end{lemma}
{\it Proof}\ \ 
For $\nu=(\nu_j)_{j\geq 1} \in\cF$ holds
$J=\max\{ j\in\mathbb{N}: \nu_j\neq 0\} < \infty$.
For this $J$, define 
$\Theta_{[J]}(z^J) := \Theta(z_1,z_2,...,z_J,0,...)$,
i.e. $\Theta_{[J]}(z^J)$ denotes the 
function of $z^J \in \mathbb{C}^J$ obtained by 
setting in the posterior density $\Theta(z)$ 
all coordinates $z_j$ with $j > J$ equal to zero. 
Then 
$$
\partial_{z}^\nu \Theta(z)|_{z=0}
=
\frac{ \partial^{|\nu|} \Theta_{[J]} }
     { \partial z_1^{\nu_1} ... \partial z_J^{\nu_J} }
(0,...,0)
\;.
$$
Since the sequence $\rho$ is $r$-admissible it follows 
with \eqref{eq:PhiArbound} that
\beq
\sup_{(z_1,...,z_J)\in \cU_{\rho,J}} |\Theta_{[J]} (z_1,\ldots,z_J)| 
\leq 
\exp
\left(
\frac{\| f \|_{V^*}^2}{r^2} \sum_{k=1}^K \| o_k \|_{V*}^2
\right)
\;.
\label{ThetaJbound}
\eeq
for all $(z_1,\ldots,z_J)$ in the polydisc
$\cU_{\rho,J}:=\otimes_{1\leq j \leq J}\{{z_j\in\mathbb{C}}: |z_j|\leq \rho_j\}
 \subset \mathbb{C}^J$.
We now prove \eqref{eq:taunubound} by Cauchy's integral formula.
To this end, we define $\tilde{\rho}$ by
$$
\tilde{\rho}_j := \rho_j + \eps  \mbox{  if  } j\leq J,\;\; \tilde{\rho}_j = \rho_j  
\mbox{  if  } 
j> J,\;\;\; \eps := \frac {r}{2\|\sum_{j\leq J}|\psi_j|\|_{L^\infty(D)}}.
$$
Then the sequence $\tilde{\rho}$ is 
$r/2$-admissible and therefore $\cU_{\tilde{\rho}} \subset \cA_{r/2}$.
This implies that for each $z\in \cU_{\tilde{\rho}}$, 
$u$ is holomorphic in each variable $z_j$.

It follows that $u_J$ is holomorphic in each variable $z_1,\ldots,z_J$ 
on the polydisc $\otimes_{1\leq j\leq J}\{ |z_j| < \tilde{\rho}_j\}$
which is an open neighbourhood of $\cU_{\rho,J}$ in $\mathbb{C}^J$.

We may thus apply the Cauchy formula 
(e.g. Theorem 2.1.2 of \cite{HoermandComplex}) in each variable $z_j$:
$$
u_J(z_1,\ldots,z_J)
=
(2\pi i)^{-J}
\int_{|\tilde{z}_1|=\tilde{\rho}_1} \ldots \int_{|\tilde{z}_J| = \tilde{\rho}_J} 
\frac {u_J(\tilde{z}_1,\ldots, \tilde{z}_J)}{(z_1 - \tilde{z}_1)\ldots(z_J - \tilde{z}_J)} 
d\tilde{z}_1\ldots d\tilde{z}_J\;.
$$
We infer
$$
\frac {\partial^{|\nu|}}{\partial z_1^{\nu_1}\ldots \partial z_J^{\nu_J}} u_J(0,\ldots,0)
=
\nu! (2\pi i)^{-J}
\int_{|\tilde{z}_1|=\tilde{\rho}_1} \ldots \int_{|\tilde{z}_J| = \tilde{\rho}_J} 
\frac{u_J(\tilde{z}_1,\ldots, \tilde{z}_J)}{\tilde{z}_1^{\nu_1}\ldots \tilde{z}_J^{\nu_J}}
d\tilde{z}_1\ldots d\tilde{z}_J\;.
$$
Bounding the integrand on 
$\{|\tilde{z}_1| = \tilde{\rho}_1\} \times \ldots \times \{|\tilde{z}_J| = \tilde{\rho}_J\} \subset \cA_r$
with \eqref{eq:PhiArbound} implies \eqref{eq:taunubound}.
\hfill $\Box$
\subsubsection{$\sigma$-summability of the $\tau_\nu$}
\label{sssec:Sumtaunu}
Proceeding in a similar fashion as in 
Section 3 of \cite{CDS2}, we can prove the 
$\sigma$-summability of the Taylor coefficients $\tau_\nu$.
\begin{proposition}\label{prop:tausum}
Under Assumptions \ref{assump1}, \ref{assump2} and \ref{summabilityofpsi},
$(\| \tau_\nu \|_V) \in \ell^\sigma(\cF)$ 
for $0 < \sigma < 1$ as in Assumption \ref{summabilityofpsi}.
\end{proposition}
We remark that under the same assumptions, we also 
have $\sigma$-summability of $(\tau_\nu/(\nu+\b1)!)_{\nu\in \cF}$, 
since 
$$
\forall \nu\in \cF:\quad
|\tau_\nu| \geq \frac{|\tau_\nu|}{(\nu+\b1)!}
\;.
$$
\subsection{Best $N$-term convergence rates}
With \eqref{stechkin},
we infer from Proposition \ref{lp} 
and from \eqref{eq:ThNRate} convergence rates 
for ``polynomial chaos'' type approximations
of the posterior density $\Th$.
\begin{theorem}\label{semierror}
If Assumptions \ref{assump1}, \ref{assump2} and 
\ref{summabilityofpsi} hold then 
there is a sequence 
$(\Lambda_N)_{N\in\mathbb{N}} \subset \cF$ 
of index sets with cardinality 
not exceeding $N$ (depending $\sigma$ and on the data $\delta$)
such that the corresponding 
$N$-term truncated gpc Legendre expansions 
$\Theta_{\Lambda_N}$ in \eqref{eq:SLambN} 
satisfy
\beq\label{eq:BestNLegRate}
\|\Th-\Th_{\Lambda_N}\|_{L^2(U,\mu_0(dy))} 
\leq
N^{-(\frac1\sigma - \frac12)}\|(\theta_{\nu})\|_{\ell^\sigma (\cF;\bbR)}
\;.
\eeq
Likewise, for $q=1,\infty$ and for every $N\in \IN$, there
exist sequences $(\Lambda_N)_{N\in \IN}\subset\cF$ 
of index sets (depending, in general, on 
$\sigma$, $q$ and the data)
whose cardinality does not exceed $N$ such that
the $N$-term truncated Taylor sums \eqref{eq:TLamN} 
converge with rate $1/\sigma - 1$, i.e. 
\beq\label{eq:BestNTayRate}
\|\Th - T_{\Lambda_N}\|_{L^q(U,\mu_0(dy))} 
\leq
N^{-(\frac1\sigma - 1)}\|(\tau_{\nu})\|_{\ell^\sigma(\cF;\bbR)}
\;.
\eeq
Here, for $q=\infty$ the norm 
$\|\circ\|_{L^\infty(U;\mu_0)}$ is
the supremum over all $y\in U$.
\end{theorem}
\section{Approximation of Expectations under the Posterior}
\label{sec:Expost}
Recall that in our approach to Bayesian estimation, the
expectations under the posterior given data $\delta$ 
are rations of deterministic, 
infinite dimensional parametric integrals $Z'$ and $Z$
with respect to the prior measure $\mu_0$, 
given by \eqref{eq:Z} and \eqref{eq:intpsi}.
For our specific elliptic inverse problem these reduce to
iterated integrals over the coordinates $y_j \in [-1,1]$ 
against a countable product of the uniform 
probability measures $\frac12 dy_j$.
To render this practically feasible,
numerical evaluation of integrals of the form 
\beq\label{eq:Ephiudelta}
\overline{\phi(u)}^\delta = \int_{y\in U} \phi(u(\cdot,y)) \Theta(y) \mu_0(dy) \in S
\eeq
are required for functions $\phi: U \to S$, for
a suitable state space $S$.
Note that the choice $\phi \equiv 1$ gives $Z$. 
For $\phi$ not identically $1$, the integral 
\eqref{eq:Ephiudelta} gives the
(posterior) conditional expectation $\IE_{\mu^\delta}[\phi(u)]$ 
if normalized by $Z$.

For the ellliptic inverse problems studied here,
the choices of $\phi(u)=u$ 
given by \eqref{eq:defmpointcorp} with $G(u)=p$
are of particular interest.
For $p=1$ this gives rise to the need to evaluate the integrals
\beq\label{eq:Epdelta}
\bar{p}^\delta 
= 
\int_{y\in U} p(\cdot,y) \Theta(y) \mu_0(dy) \in V
\eeq
which, when normalized by $Z$, gives the
(posterior) conditioned expectation $\IE_{\mu^\delta}[p].$
We study how to approximate this integral.
With the techniques developed here,
and with Corollary \ref{cor:PsiAnalyt},
analogous results can also be established for 
expectations of $m$ point correlations of 
$G(u)$ as in \eqref{eq:defmpointcorp}, using \eqref{eq:Ephiudelta},
and the normalization constant $Z$.

Our objective is to find constructive algorithms
which achieve the high rates of convergence, in terms 
of number of retained terms $N$ in a gpc expansion,
implied by the theory of the previous section, and
offering the potential of beating the complexity of
Monte Carlo based methods.
The first option to do so is to employ 
{\em sparse tensor numerical integration
scheme over $U$} tailored to the 
regularity afforded by the analytic parameter
dependence of the posteriori density on $y$
and of the integrands in  
\eqref{eq:Ephiudelta}.
This approach is not considered here, but is
considered elsewhere: we
refer to \cite{AScSt} for details and 
numerical experiments.
Here we adopt an approach based on showing that the integrals 
\eqref{eq:Ephiudelta}
allow {\em semianalytic evaluation} in 
log-linear\footnote{Meaning linear multiplied by a logartihmic factor.}
complexity with respect to $N$, the number of
``active'' 
terms in a 
truncated polynomial chaos expansion
of the parametric solution of the forward
problem \eqref{eq:fwdproblem}, \eqref{eq:par}.

To this end, we proceed as follows: based on 
the {\em assumption} that $N$-term gpc approximations
of the parametric forward solutions $p(x,y)$ of 
\eqref{eq:fwdproblem} is available, for example
by the algorithms in \cite{BAS09,CJG11,CCDS11},
we show that it is possible to 
{\em construct separable $N$-term approximations}
of the integrands in 
\eqref{eq:Ephiudelta}.
The existence of such an approximate posterior density 
which is ``close'' to $\Th$ is ensured by Theorem \ref{semierror}, 
provided the (unknown) input data $u$ satisfies certain conditions.
{\em 
We prove that sets $\Lambda_N\subset \cF$ of 
cardinality at most $N$ which afford the truncation errors 
\eqref{eq:BestNLegRate}, \eqref{eq:BestNTayRate} 
can be found in log-linear complexity with respect to $N$} and, 
second, 
{\em that the integrals \eqref{eq:Ephiudelta} with the 
corresponding approximate posterior density 
can be evaluated in such complexity}
and, third, we estimate the errors in the resulting 
conditional expectations.
\subsection{Assumptions and Notation}
\label{ssec:AssNotation}
\begin{ass}\label{ass:exactp}
Given a draw $u$ of the data, 
an exact forward solution $p$ 
of the governing equation \eqref{eq:fwdproblem} 
{\em for this draw of data $u$}
is available at unit cost.
\end{ass}
This assumption is made 
in order to simplify the exposition.
All conclusions remain valid if this assumption is 
relaxed to include an additional 
Finite Element discretization error;
we refer to \cite{AScSt} for details.
We shall use the notion of 
{\em monotone sets of multiindices}.
\begin{definition} \label{def:MonotIndx}
A subset $\Lambda_N \subset \cF$ of finite
cardinality $N$ is called {\em monotone}
if 
(M1) $\{0\}\subset \Lambda_N$ 
and if 
(M2) $\forall 0\ne \nu \in \Lambda_N$ 
it holds that $\nu-e_j \in \Lambda_N$
for all $j\in  \bbI_{\nu}$, where
$e_j\in \{0,1\}^\bbJ$ denotes the index vector
with $1$ in position $j\in \bbJ$ and 
$0$ in all other positions $i \in \bbJ\backslash\{j\}$.
\end{definition}
Note that for 
monotone index sets $\Lambda_N\subset \cF$ 
properties (M1) and (M2) 
in Definition \ref{def:MonotIndx} imply
\beq\label{spynu=spnu}
\IL_{\Lambda_N}(U) 
= {\rm span}\{ y^\nu : \nu\in \Lambda_N \}
= {\rm span}\{ L_\nu : \nu\in \Lambda_N \}
\;.
\eeq
Next, we will assume that 
{\em 
a stochastic Galerkin approximation
of the entire forward map of the parametric,
deterministic solution 
with certain optimality properties
is available.
}
\begin{ass}\label{ass:sGFEM}
Given a parametric representation 
\eqref{assume1} of the unknown data $u$, 
a stochastic Galerkin approximation 
$p_N \in \IL_{\Lambda_N}(U,V)$ 
of the exact forward solution 
of the governing equation
\eqref{eq:fwdproblem} is available at unit cost.
Here the set $\Lambda_N\subset \cF$ 
is a finite subset of ``active'' 
gpc Legendre coefficients whose 
cardinality does not exceed $N$.
In addition, we assume that 
the gpc approximation 
$p_N \in \IL_{\Lambda_N}(U,V)$
is quasi optimal in terms of the 
best $N$-term approximation, i.e. 
there exists $C\geq 1$ independent 
of $N$ such that
\beq\label{eq:uNbestN}
\| p - p_N \|_{L^2(U,\mu_0;V)}
\leq
C N^{ -(1/\sigma  - 1/2)} 
\| (\theta_\nu) \|_{\ell^\sigma(\cF)}
\;.
\eeq
Here $0 < \sigma \leq 1$ 
denotes the summability exponent in 
Assumption \ref{summabilityofpsi}.
Note that best $N$-term approximations
satisfy \eqref{eq:uNbestN} with $C=1$;
we may refer to \eqref{eq:uNbestN} as a
quasi best $N$-term approximation property.
\end{ass}
This best $N$-term convergence rate of 
stochastic Galerkin Finite Element Method (sGFEM) 
approximations follows from results in 
\cite{CDS1,CDS2}, but these results do not 
indicate as to how sequences of sGFEM 
approximations which converge with this
rate are actually constructed.
We refer 
to \cite{CJG11} for the constructive algorithms 
for quasi best $N$-term 
Legendre Galerkin approximations and 
to \cite{CCDS11} for constructive algorithms 
for quasi best $N$-term Taylor approximations
and also to the references there
for details on further details for such sGFEM solvers, 
including space discretization.
In what follows, we work under 
Assumptions \ref{ass:exactp}, \ref{ass:sGFEM}.
\subsection{Best $N$-term based approximate conditional expectation}
\label{ssec:BestNCondEx}
We first address the rates that can be achieved
by the (a-priori not accesssible) best $N$-term
approximations of the posterior density $\Th$ in 
Theorem \ref{semierror}. 
These rates serve as benchmark rates to be achieved
by any constructive procedure.

To derive these rates, we let 
$\Th_N  = \Th_{\Lambda_N}$
denote the best $N$-term Legendre approximations
of the posterior density $\Th$ in Theorem 
\ref{semierror}. 
With \eqref{eq:uNbestN}, we estimate
$$
\begin{array}{rcl}
\| \bar{p}^\delta - \bar{p}_N^\delta \|_V
& = & \ds
\left\| \int_U \left( \Th p - \Th_N p_N \right) \mu_0(dy) \right\|_V
\\
& = & \ds
\left\| 
\int_U 
\left( (\Th - \Th_N) p + \Th_N ( p-p_N) \right) \mu_0(dy) \right\|_V
\\
& \leq & \ds
\int_U |\Th - \Th_N | \| p \|_V \mu_0(dy) 
+
\| \Th_N \|_{L^2(U)} \| p - p_N \|_{L^2(U,\mu_0;,V)}
\\
& \leq & \ds
\| \Th - \Th_N \|_{L^2(U)} \| p \|_{L^2(U,\mu_0;V)}
+
\| \Th_N \|_{L^2(U)} \| p - p_N \|_{L^2(U,\mu_0;V)}
\\
& \leq & \ds
C N^{-(\frac{1}{\sigma} - \frac12)}
\;.
\end{array}
$$
With $T_N = T_{\Lambda_N}$ denoting a best $N$-term Taylor 
approximation of $\Th$ in Theorem \ref{semierror} 
we obtain in the same fashion the bound
$$
\begin{array}{rcl}
\| \bar{p}^\delta - \bar{p}_N^\delta \|_V
& = & \ds
\left\| \int_U \left( \Th p - T_N p_N \right) \mu_0(dy) \right\|_V
\\
& = & \ds
\left\| 
\int_U 
\left( (\Th - T_N) p + T_N ( p-p_N) \right) \mu_0(dy) \right\|_V
\\
& \leq & \ds
\int_U |\Th - T_N | \| p \|_V \mu_0(dy) 
+
\| T_N \|_{L^\infty(U)} \| p - p_N \|_{L^1(U,\mu_0;V)}
\\
& \leq & \ds
\| \Th - T_N \|_{L^1(U,\mu_0)} \| p \|_{L^\infty(U,\mu_0;V)}
+
\| T_N \|_{L^\infty(U)} \| p - p_N \|_{L^2(U,\mu_0;V)}
\\
& \leq & \ds
C N^{-(\frac{1}{\sigma} - 1)}
\;.
\end{array}
$$

We now address question ii) raised at the beginning of Section
\ref{ssec:BestNTh}, i.e. 
{\em 
the design of practical algorithms for the construction of 
sequences 
$(\Lambda_N)_{N\in \IN} \subset \cF$
such that the best-$N$ term convergence rates
asserted in Theorem \ref{semierror} are attained.
} 
We develop the approximation in detail for \eqref{eq:Epdelta}; 
similar results for \eqref{eq:Ephiudelta} may be developed
for various choices of $\phi$.
\subsection{Constructive $N$-term Approximation of the Potential $\Phi$}
\label{ssec:ApproxPot}
We show that, from the quasi best $N$-term optimal
stochastic Galerkin approximation 
$u_N\in \IL_{\Lambda_N}(U,V)$ 
and, in particular,
from its (monotone) index set $\Lambda_N$, a corresponding
$N$-term approximation $\Phi_N$ of the potential $\Phi$
in \eqref{eq:lsq} can be computed.
We denote the observation corresponding to the
stochastic Galerkin approximation of the 
system response $p_N$ by $\cG_N$, i.e.
the mapping
\beq \label{eq:defcGN}
U\ni y\mapsto \cG_N(u)|_{u=\bar a + \sum_{j\in \bbJ} y_j\psi_j} 
= 
(\cO\circ G_N)(u)|_{u=\bar a + \sum_{j\in \bbJ} y_j\psi_j}
\eeq
where $G_N(u) = p_N\in \IL_{\Lambda_N}(U;V)$. 
By the linearity and boundedness of the observation
functional $\cO(\cdot)$ then 
$\cG_N \in \IL_{\Lambda_N}(U;\IR^K)$; 
in the following, we assume for simplicity $K=1$ 
so that 
$\cG_N|_{u=\bar a + \sum_{j\in \bbJ} y_j\psi_j} \in \IL_{\Lambda_N}(U)$.
We then denote by $U\ni u\mapsto \Phi$ the potential 
in \eqref{eq:lsq} and by $\Phi_N$ the potential of the stochastic
Galerkin approximation $\cG_N$ of the forward observation map.
For notational convenience, we suppress the explicit dependence
on the data $\delta$ in the following and assume that the 
Gaussian covariance $\Gamma$ of the observational noise $\eta$
in \eqref{eq:obs1} is the identity: $\Gamma = I$. 
Then, for every $y\in U$, 
with $u =\bar a + \sum_{j\in \bbJ} y_j \psi_j$
the exact potential $\Phi$ and the potential
$\Phi_N$
based on $N$-term approximation $p_N$ of the 
forward solution take the form
\beq \label{eq:cGcGN}
\Phi(y)   = \frac{1}{2} (\delta - \cG(u))^2,
\quad
\Phi_N(y) = \frac{1}{2} (\delta - \cG_N(u))^2
\;.
\eeq
By Lemma \ref{l:two}, these potentials admit
extensions to holomorphic functions of the variables 
$z\in \cS_\rho$ in the strip $\cS_\rho$ 
defined in \eqref{eq:Sprod}. 
Since $\Lambda_N$ is monotone, we may write 
$p_N\in \IL_{\Lambda_N}(U,V)$ and 
$\cG_N \in \IL_{\Lambda_N}(U)$ in terms of their
(uniquely defined) Taylor expansions about $y=0$:
\beq\label{eq:cGNTaylor}
\cG_N(u) = \sum_{\nu\in \Lambda_N} g_\nu y^\nu \;.
\eeq
This implies, for every $y\in U$,
$ \Phi_N(y) = \delta^2 -2\delta\cG_N(y) + (\cG_N(y))^2$
where 
$$
(\cG_N(y))^2
=
\sum_{\nu,\nu' \in \Lambda_N} g_\nu g_{\nu'} y^{\nu+\nu'}
\in 
\IL_{\Lambda_N+\Lambda_N}(U)
$$
has a higher polynomial degree and possibly $O(N^2)$ 
coefficients.
Therefore, an exact evaluation of a gpc approximation
of the potential $\Phi_N$ might incur 
loss of linear complexity with respect to $N$. 
To preserve log-linear in $N$ complexity, we
perform an $N$-term 
truncation $[\Phi_N]_{\# N}$ of $\Phi_N$,
thereby introducing an additional error which, as 
we show next, is of the same order as the error
of gpc approximation of the system's response.
The following Lemma is stated in slightly  more 
general form than is presently needed, since it will
also be used for the error analysis of the posterior
density ahead.
\begin{lemma}\label{lem:truncN}
Consider two sequences 
$(g_\nu)\in \ell^\sigma(\cF)$, 
$(g'_{\nu'})\in \ell^\sigma (\cF')$, 
$0 < \sigma \leq 1$.
Then
$$
(g_\nu g'_{\nu'})_{(\nu,\nu')\in \cF\times \cF'}
\in 
\ell^\sigma (\cF\times \cF')
$$
and there holds
\beq\label{eq:prodellpest}
\| (g_\nu g'_{\nu'}) \|^\sigma_{\ell^\sigma(\cF\times \cF')}
\leq
\| (g_\nu)     \|^\sigma_{\ell^\sigma (\cF)}
\| (g'_{\nu'}) \|^\sigma_{\ell^\sigma (\cF')}
\;.
\eeq
Moreover, a best $N$-term truncation 
$[ \circ ]_{\#}$
of products of corresponding  best
$N$-term truncated Taylor polynomials, defined by 
\beq \label{eq:bestN}
\left[
\left(
\sum_{\nu\in \Lambda_N} g_\nu y^\nu
\right)
\left(
\sum_{\nu'\in \Lambda'_N} g'_{\nu'} y^{\nu'}
\right)
\right]_{\# N}
:=
\sum_{ (\nu,\nu') \in \Lambda^1_N}
g_\nu g'_{\nu'} y^{\nu+\nu'}
\in 
\IL_{\Lambda_N^1}(U)
\eeq
where $\Lambda^1_N\subset \cF\times \cF'$ 
is the set of sums of index pairs 
$(\nu,\nu')\in \cF\times \cF'$
of at most $N$ largest (in absolute value) 
products $g_\nu g_{\nu'}$, 
has a pointwise error in $U$ bounded by
\beq\label{eq:proderrlp}
N^{-(\frac1\sigma - 1)} \|(g_\nu)\|_{\ell^\sigma(\cF)} \|(g'_{\nu'})\|_{\ell^\sigma(\cF')} 
\;.
\eeq
Moreover, if the index sets 
$\Lambda_N\subset \cF$ and $\Lambda'_N\subset\cF'$
are each monotone, the index set 
$\bar{\Lambda}_N := \{ \nu+\nu': (\nu,\nu')\in \Lambda^1_N \}\subset \cF$
can be chosen monotone with cardinality at most $2N$.
\end{lemma}
\begin{proof}
We calculate
$$
\begin{array}{rcl}
\ds
\| g_\nu g'_{\nu'} \|^\sigma_{\ell^\sigma(\cF\times\cF)}
& = & \ds 
\sum_{\nu\in \cF} \sum_{\nu'\in \cF} |g_\nu g'_{\nu'} |^\sigma
= 
\sum_{\nu\in \cF} \left( |g_\nu|^\sigma \sum_{\nu'\in \cF} |g'_{\nu'}|^\sigma \right)
\\
& = & \ds 
\|(g_\nu)\|^\sigma_{\ell^\sigma(\cF)} \|(g'_{\nu'})\|^\sigma_{\ell^\sigma(\cF)} 
\;.
\end{array}
$$
Since $(g_\nu g'_{\nu'}) \in \ell^\sigma(\cF\times \cF)$,
we may apply \eqref{stechkin} with \eqref{eq:prodellpest}
as follows.
$$
\begin{array}{l}
\ds
\left\|
\left[
\sum_{\nu\in \Lambda_N} \sum_{\nu'\in \Lambda'_N}
g_\nu g'_{\nu'} y^{\nu'+\nu}
\right]
-
\left[
\sum_{\nu\in \Lambda_N} \sum_{\nu'\in \Lambda'_N}
g_\nu g'_{\nu'} y^{\nu'+\nu}
\right]_{\# N}
\right\|_{L^\infty(U)}
\\ \mbox{ } \\
\leq 
\ds
\sum_{(\nu,\nu')\in \cF\times \cF \backslash \Lambda^1_N}
| g_\nu g'_{\nu'} | 
\leq  \ds
N^{-(\frac1\sigma - 1)} \|(g_\nu)\|_{\ell^\sigma(\cF)} \|(g'_{\nu'})\|_{\ell^\sigma(\cF)} 
\;.
\end{array}
$$
Evidently, $\bar{\Lambda}_N \subseteq \Lambda_N + \Lambda'_N$
and the cardinality of the set $\Lambda_N + \Lambda'_N$ is at most $2N$.
If $\Lambda_N$ and $\Lambda'_N$ are monotone,
then $\Lambda_N + \Lambda'_N$ is monotone. 
To see it, let $\mu \in \Lambda_N + \Lambda'_N$. 
Then 
$\mu = \nu + \nu'$ for some $\nu\in\Lambda_N$, 
$\nu'\in \Lambda'_N$, and 
$\bbI_\mu = \bbI_{\nu} \cup \bbI_{\nu'}$.
Let $0\ne \mu$, $j\in \bbI_\mu$ and assume 
w.l.o.g. that $j\in \bbI_\nu$.
Then 
$\mu - e_j = (\nu-e_j) + \nu' \in \Lambda_N + \Lambda'_N$
by the assumed monotonicity of the set $\Lambda_N$.
If $j\in \bbI_{\nu'}$, the argument is analogous.
Therefore $\mu-e_j \in \Lambda_N + \Lambda'_N$ 
for every $j\in \bbI_\mu$. 
Hence $\Lambda_N + \Lambda'_N\subset \cF$ is monotone.
\end{proof}
Lemma \ref{lem:truncN} is key to the analysis
of consistency errors in the approximate evaluation
of $N$-term truncated power series and, in particular,
of the potential $\exp(-\Phi(u;\delta))$ which appears
in the posterior density $\Theta$. It crucially 
involves Taylor-type polynomial chaos expansions.
Expansions based on Legendre (or other) univariate
polynomial bases can be covered by Lemma \ref{lem:truncN}
by conversion to monomial bases, using \eqref{spynu=spnu}, 
as long as $N$-term truncations are restricted to 
monotone index sets $\Lambda_N \subset \cF$.

Applying Lemma \ref{lem:truncN} with $\cF' = \cF$
and with $(g'_{\nu'})_{\nu'\in \cF'} = (g_\nu)_{\nu\in \cF}$, 
we find

\beq\label{eq:truncNPhiN}
\begin{array}{rcl}
\ds
\sup_{y\in U} 
\left|
\Phi_N(y) - \left[ \Phi_N(y) \right]_{\#N}
\right|
&=& \ds
\sup_{y\in U}
\left|
(\cG_N(y))^2 - \left[(\cG_N(y))^2\right]_{\#N}
\right|
\\
&\leq& \ds 
N^{-(\frac1\sigma - 1)} \|(g_\nu)\|_{\ell^\sigma (\cF)}^2
\;.
\end{array}
\eeq
%
%
\subsection{Constructive $N$-term approximation of $\Th=\exp(-\Phi)$}
\label{ssec:ConstNTermTh}
With the $N$-term approximation $[\Phi_N]_{\#N}$, we 
now define the 
{\em constructive $N$-term 
     approximation $\Th_N$ of the posterior density}.
We continue to work under 
Assumption \ref{ass:sGFEM}, i.e. 
{\em 
that $N$-term 
truncated gpc-approximations $p_N$ 
of the forward solution $p(y)=G(u(y))$ 
of the parametric problem are available which 
satisfy \eqref{eq:uNbestN}.
}
For an integer $K(N)\in \bbN$ to be selected
below, we define
\beq\label{eq:Thconstr}
\Th_N 
=
\sum_{k=0}^{K(N)}
\frac{(-1)^k}{k!} \left[ ([\Phi_N]_{\#N}])^k \right]_{\#N} 
\;.
\eeq
We then estimate (all integrals are 
with respect to the prior measure $\mu_0(dy)$)
$$
\begin{array}{l}
\ds
\| \Th - \Th_N \|_{L^1(U)}
= 
\left\| 
e^{-\Phi} - e^{-[\Phi_N]_{\#N}} + e^{-[\Phi_N]_{\#N}} 
-
\sum_{k=0}^{K(N)}
\frac{(-1)^k}{k!} \left[ ([\Phi_N]_{\#N}])^k \right]_{\#N}
\right\|_{L^1(U)}
\\
\ds
\leq 
\left\| 
e^{-\Phi} - e^{-[\Phi_N]_{\#N}}
\right\|_{L^1(U)}
+ 
\left\| 
e^{-[\Phi_N]_{\#N}} 
-
\sum_{k=0}^{K(N)}
\frac{(-1)^k}{k!} \left[ ([\Phi_N]_{\#N}])^k \right]_{\#N}
\right\|_{L^1(U)}
\\
\ds
=: I + II \;.
\end{array}
$$
We estimate both terms separately.

For term $I$, we observe that due to 
$x=[\Phi_N]_{\#N} - \Phi \geq 0$
for sufficiently large values of $N$,
it holds $0\leq 1 - e^{-x} \leq x$, 
so that by the triangle inequality and 
the bound \eqref{eq:truncNPhiN}
$$
\begin{array}{rcl}
I &=& \ds
 \left\| e^{-\Phi}(1-e^{\Phi - [\Phi_N]_{\#N}}) \right\|_{L^1(U)} 
\leq 
  \left\| \Th \right\|_{L^\infty(U)}  
  \left\| 1 - e^{-([\Phi_N]_{\#N} - \Phi)}\right\|_{L^1(U)}
\\
& \leq & \ds
 \left\| \Th \right\|_{L^\infty(U)} 
 \left\| \Phi - [\Phi_N]_{\#N} \right\|_{L^1(U)}
\leq 
C \left(
\left\| \Phi - \Phi_N \right\|_{L^1(U)} 
+
\left\| \Phi_N - [\Phi_N]_{\#N} \right\|_{L^1(U)}
\right)
\\
& \leq & \ds
\left\| p-p_N \right\|_{L^2(U,V)} 
+
CN^{-(\frac1x\sigma - 1)} 
\leq
C N^{-(\frac1\sigma - 1)}
\end{array}
$$
where $C$ depends on $\delta$, but is independent of $N$.
In the preceding estimate,
we used that 
$\Phi > 0$ and $0\leq \Th = \exp(-\Phi) < 1$
imply
$$
\left\| \Phi - \Phi_N \right\|_{L^1(U)}
\leq 
\| \cO \|_{V^*} \| p - p_N \|_{L^2(U,V)}
\left(
2|\delta| + \|\cO\|_{V^*} \| p+p_N \|_{L^2(U,V)}
\right)
\;.
$$
We turn to term $II$. Using the (globally convergent)
series expansion of the exponential function, 
we may estimate with the triangle inequality
\beq \label{eq:EstII}
\begin{array}{rcl}
II & \leq & \ds
\left\| R_{K(N)} \right\|_{L^1(U)}
+
\sum_{k=0}^{K(N)} 
\frac{1}{k!}
\left\| 
([\Phi_N]_{\#N})^k
-
\left[ ([\Phi_N]_{\#N})^k \right]_{\#N}
\right\|_{L^1(U)}
\end{array}
\eeq
where the remainder $R_{K(N)}$ equals
\beq\label{eq:defRKN}
R_{K(N)} 
= 
\sum_{k=K(N)+1}^\infty \frac{(-1)^k}{k!} ([\Phi_N]_{\#N}])^k
\;.
\eeq
To estimate the second term in the bound \eqref{eq:EstII}
we claim that for every $k,N\in \bbN_0$ holds
\beq\label{eq:kbound}
\left\| 
([\Phi_N]_{\#N})^k 
- 
\left[ ([\Phi_N]_{\#N})^k \right]_{\#N}
\right\|_{L^\infty(U)}
\leq
N^{-(\frac1\sigma - 1)} \| (g_\nu) \|^{2k\sigma}_{\ell^\sigma(\cF)}
\;.
\eeq
We prove \eqref{eq:kbound} for arbitrary, 
fixed $N\in \bbN$ by induction with respect to $k$.
For $k=0,1$, the bound is obvious.
Assume now that the bound has been established for all
powers up to some $k\geq 2$.
Writing 
$([\Phi_N]_{\#N})^{k+1} = ([\Phi_N]_{\#N})^k [\Phi_N]_{\#N}$
and denoting the sequence of Taylor coefficients of
$[\Phi_N]^k$ by $g'_{\nu'}$ with 
$\nu' \in (\cF\times \cF)^k \simeq \cF^{2k}$,
we note that by $k$-fold application of 
\eqref{eq:prodellpest} it follows
$\| (g'_{\nu'}) \|^\sigma_{\ell^\sigma(\cF^{2k})} 
 \leq 
 \|(g_\nu)\|^{2k\sigma}_{\ell^\sigma(\cF)}$.
By the definition of $[\Phi_N]_{\#N}$,
the same bound also holds for the coefficients of 
$([\Phi_N]_{\#N})^k$, for every $k\in \bbN$.
We may therefore apply Lemma \ref{lem:truncN} 
to the product $([\Phi_N]_{\#N})^k [\Phi_N]_{\#N}$
and obtain the estimate \eqref{eq:kbound} with
$k+1$ in place of $k$ from \eqref{eq:proderrlp}.
Inserting \eqref{eq:kbound} into \eqref{eq:EstII}, 
we find
\beq\label{eq:EstII'}
\begin{array}{rcl}
\ds
\sum_{k=0}^{K(N)} 
\frac{1}{k!}
\left\| 
([\Phi_N]_{\#N})^k
-
\left[ ([\Phi_N]_{\#N})^k \right]_{\#N}
\right\|_{L^1(U)}
&\leq &\ds
N^{-(\frac1\sigma  - 1)} 
\sum_{k=0}^{K(N)} \frac{1}{k!} \| (g_\nu) \|_{\ell^\sigma(\cF)}^{2k\sigma}
\\
&\leq & \ds
N^{-(\frac1\sigma - 1)} 
\exp(\| (g_\nu) \|_{\ell^\sigma(\cF)}^{2\sigma})
\;.
\end{array}
\eeq
In a similar fashion, we estimate the remainder 
$R_{K(N)}$ in \eqref{eq:EstII}: 
as the truncated Taylor expansion $[\Phi_N]_{\#N}$
converges pointwise to $\Phi_N$ and to $\Phi > 0$,
for sufficiently large $N$, we have
$[\Phi_N]_{\#N}>0$ for all $y\in U$, so that the series
\eqref{eq:defRKN} is alternating and 
converges pointwise. Hence its truncation error is 
bounded by the leading term of the tail sum:
\beq \label{eq:RKNest}
\| R_{K(N)} \|_{L^\infty(U)} 
\leq 
\frac{\| [\Phi_N]_{\#N} \|_{L^\infty(U)}^{K(N)+1}}{(K(N)+1)!} 
\leq 
\frac{ \| (g_\nu) \|_{\ell^1(\cF)}^{2(K(N)+1)}}{(K(N)+1)!} 
\eeq
Now, given $N$ sufficiently large, we choose $K(N)$ 
so that the bound \eqref{eq:RKNest} is smaller than
\eqref{eq:EstII'}, which leads with Stirling's 
formula in \eqref{eq:RKNest} to the requirement
\beq \label{eq:Kchoice}
(K+1) \ln\left(\frac{Ae}{K}\right) \leq \ln B - (\frac1\sigma -1)\ln N
\eeq
for some constants $A, B>0$ independent of $K$ and $N$ 
(depending on $p$ and on $(g_\nu)$).
One verifies that \eqref{eq:Kchoice} is satisfied by
selecting $K(N) \simeq  \ln N$. 

Therefore, under Assumptions \ref{ass:exactp} and \ref{ass:sGFEM},
we have shown how to construct 
an $N$-term approximate posterior density $\Th_N$
by summing $K = O(\ln N)$ many terms in \eqref{eq:Thconstr}.
The approximate posterior density has at most $O(N)$ nontrivial
terms, which can be integrated exactly against the separable
prior $\mu_0$ over $U$ in complexity that behaves log-linearly
with respect to $N$, under 
Assumptions \ref{ass:exactp}, \ref{ass:sGFEM}:
the construction of $\Th_N$ requires $K$-fold performance of the 
$[\cdot]_{\#N}$-truncation operation in \eqref{eq:bestN}
of products of Taylor expansions, 
with each factor having at most $N$ nontrivial entries,
amounting altogether to solving (possibly approximately)
$
O(K N \ln N) = O(N (\ln N)^2)
$
forward problems.
\begin{remark} \label{remk:monotone}
Inspecting the (constructive) proof of Lemma \ref{lem:truncN}
and the definition of the $N$-term approximation $\Theta_N$
of the posterior density \eqref{eq:Thconstr}, we see that 
the index set $\Lambda_N^\Theta$ of active Taylor gpc 
coefficients of $\Theta_N$ satisfies 
$$
\Lambda_N^\Theta
\subset 
\overline{\Lambda_N^\Theta}
:=
(\Lambda_N + \Lambda_N) + ... (K(N)-{\rm times})...+(\Lambda_N+\Lambda_N)
\subset \cF
$$
where $\Lambda_N\subset \cF$ is the set of 
$N$ active gpc coefficients in the approximate 
forward solver in Assumption \ref{ass:sGFEM}.

If, in particular, $\Lambda_N$ is monotone, so is the set 
$\overline{\Lambda_N^\Theta}$. This follows by induction 
over $K$ with the argument in the last part of the proof of 
Lemma \ref{lem:truncN}. Moreover, 
the cardinality of $\Lambda^\Theta_N$ 
is bounded by $2NK(N) \lesssim N\log(N)$.
\end{remark}

\section{Conclusions}
\label{sec:ConGen}

This paper is concerned with formulation of Bayesian
inversion as a problem in infinite 
dimensional parametric integration, and the construction
of algorithms which exploit analyticity of the forward
map from state space to data space to approximate
these integration problems. In this section 
we make some concluding remarks about the implications
of our analysis. We discuss computational complexity
for such problems, and we discuss further directions
for research.

\subsection{Computational Cost: Idealized Analysis} 
Throughout we have been guided by the desire to
create algorithms which outperform Monte Carlo
based methods. To gain insight into this issue
we first proceed under the (idealized) setting
of Assumptions \ref{ass:exactp} and \ref{ass:sGFEM}, which
imply that the PDE \eqref{eq:fwdproblem}, 
for fixed parameter $u$,
and its parametric solution, for all $u \in U$, can both
be approximated at unit cost. 
In this situation we can study the {\em cost per unit error}
of Monte Carlo and gpc methods as follows. We neglect
logarithmic corrections for clarity of exposition. 
The Monte Carlo method will require ${\cal O}(N)$ work
to achieve an error of size $N^{-\frac12}$, where $N$
is a number of samples from the prior. To obtain error
$\epsilon$ thus requires work of order ${\cal O}(\epsilon^{-2}).$
Recall the parameter $\sigma$ from
Assumption \ref{summabilityofpsi} which
measures the rate of decay of the input fluctations and,
as we have shown, governs the smoothness properies
of the analytic map from unknown to data.
The gpc method based on best $N$ term approximation
requires work which is linear in $N$ to obtain an error
of size $N^{-(1/\sigma - 1)}$. Thus to
obtain error $\epsilon$ requires work of order 
${\cal O}(\epsilon^{\sigma/(1-\sigma)}).$ 
For all $\sigma<2/3$ the complexity of the new
gpc methods, under our idealized assumptions,
is superior to that of Monte Carlo based methods.

\subsection{Computational Cost: Practical Issues} 

The analysis of the previous subsection
provides a clear way to
understand the potential of the methods introduced
in this paper and is useful for communicating the
central idea. However, by working under the
stated Assumptions \ref{ass:exactp} and \ref{ass:sGFEM}, 
some aspects of the true computational complexity 
of the problem are hidden.
In this subsection we briefly 
discuss further issues that arise. 
Throughout we 
assume that the desired form of the unknown diffusion
coefficient for the forward PDE \eqref{eq:fwdproblem}
is given by \eqref{assume1} in the case where
$\bbJ=\bbN:$
\beq
\label{assume1new}
u(x,y)=\bar a(x)+ \sum_{j\in \bbN} y_j\psi_j(x),\quad x\in D.
\eeq
To quantify the complexity of the problem we assume that, for some $b>0$,
\begin{equation}
\label{eq:decay}
\|\psi_j\|_{L^{\infty}(D)} \asymp j^{-(1+b)}.
\end{equation}
Then Assumption \ref{summabilityofpsi} holds for any $\sigma>(1+b)^{-1}.$
In practice, to implement either Monte Carlo or gpc
based methods it is necessary to 
truncate the series \eqref{assume1new} to $J$ terms to obtain
\beq
\label{assume1m}
\um(x,y)=\bar a(x)+ \sum_{1 \le j \le J} y_j\psi_j(x),\quad x\in D.
\eeq
To quantify the computational cost of the problem we assume
that the non-parametric forward problem \eqref{eq:fwdproblem} 
with fixed $u \in U$, incurs costs $\cc(J,\epsilon)$ 
to make an error of size $\epsilon$ in $V$. 
Likewise we assume that the parametric forward
problem \eqref{eq:fwdproblem}, for all $u \in U$,
incurs costs $\cp(N,J,\epsilon)$ to make an error
of $\epsilon$ in $L^2(U,\mu_0(dy);V)$ 
via computation of an approximation to a quasi-optimal 
best $N$ term gpc approximation.

Both Monte Carlo based and gpc based
methods will incur an error caused by truncation 
to $J$ terms. Using the Lipschitz
property of $\cG$ expressed in \eqref{eq:LipcG}, together
with the  arguments developed in \cite{CDS10}\footnote{The key 
idea in \cite{CDS10} is that
error in the forward problem transfers to error in the Bayesian inverse
problem, as measured in the Hellinger metric and hence for a wide
class of expectations; the analysis in \cite{CDS10} is devoted to
Gaussian priors and situations where the Lipschitz constant of
the forward model depends on the realization of the input data $u$
and Fernique theorem is used to control this dependence; this
is more complex than required here, because the
Lipschitz constants in \eqref{eq:LipcG} here do not depend on the
realization of the input data $u$. 
For these reasons we
do not feel it is necessary to provide a proof of the error incurred
by truncation.}
we deduce that the error in computing expectations caused by truncation 
of the input data to $J$ terms is proportional to
$$\sum_{j=J+1}^\infty \|\psi_j\|_{L^{\infty}(D)}.$$
Under assumption \eqref{eq:decay} this is of order 
${\cal O}(J^{-b})$ and since $b$ may be chosen arbitrarily
close to $1/\sigma-1$ we obtain an error 
${\cal O}(J^{1-1/\sigma})$ from truncation. 

The total error for Monte Carlo based methods using 
$N$ samples is then of the form
$$E_{\rm {mc}}=\frac{C(J)}{N^{\frac12}}+{\cal O}(J^{1-1/\sigma})+
\epsilon$$
In the case where $C(J)$ is independent of $J$, which arises
for pure Monte Carlo methods based on prior sampling and
for the independence MCMC sampler \cite{Liu,RC}, 
choosing $N$ and $J$ to balance the error gives
$N={\cal O}(\epsilon^{-2})$ and $J={\cal O}(\epsilon^{-\sigma/(1-\sigma)})$ and, with these relationships imposed, the
cost is $N \times \cc(J,\epsilon)$ since one forward PDE solve is made
at each step of any Monte Carlo method.
In practice standard Monte Carlo sampling may be ineffective,
because samples from the prior are not well-distributed
with respect to the posterior density; this is especially
true for problems with large numbers of observations and/or
small observational noise. In this case MCMC methods may
be favoured and it is possible that $C(J)$ will grow with $J$;
see \cite{RSwhen} for an analysis of this effect for
random walk Metropolis algorithms. Balancing the error
terms will then lead to a further increase in computational
cost.

For gpc methods based on $N$ term truncation the error
is of the form
$$E_{\rm gpc}={\cal O}(N^{1-1/\sigma})
+{\cal O}(J^{1-1/\sigma})+\epsilon$$
implying that $N=J={\cal O}(\epsilon^{-\sigma/(1-
\sigma)})$ to balance errors. This expressions must be
substituted into $\cp(N,J,\epsilon)$ to deduce the
asymptotic cost.

In practice, however, the gpc methods can also suffer
when the number of observed data is high, or when the
observational noise is small. To see this, note that the
choice of active terms in the expansion \eqref{unu}
is independent of the data, and is determined by the
prior. For these reasons it may be computationally
expedient in practice to study methods which marry
MCMC and gpc \cite{MaNajmLahnJCP07,MaXiu09,Ma08}. 
In a forthcoming paper \cite{HaSS12} we will
investigate the performance
of the gpc-based posterior approximations, in particular
in the case of values of $\sigma$ which are close to $\sigma=1$,
i.e. in the case of little or no sparsity in the 
expansion of the unknown $u$,
for parametric precomputation of an approximation of 
the law of the forward model, removing the necessity
to compute a forward solution at each step, and 
by extending this idea further to Multi-Level LMCMC.

\subsection{Outlook}
We have proved that for a class of inverse diffusion problems with 
unknown diffusion coefficient $u$, that in the context of a Bayesian
approach to the solution of these inverse problems, given the data 
$\delta$, for a class of diffusion coefficients $u$ which
are spatially heterogeneous and uncertainty parametrized by 
a countable number of random coordinate variables, 
{\em 
sparsity in the gpc expansion of $u$ entails 
the same sparsity in the density of the Bayesian posterior
with respect to the prior measure.
}

We have provided a constructive proof of
{\em 
how to obtain an approximate posterior density
by an $O(N)$ term truncated gpc expansion,
based on a set $\Lambda_N\subset \cF$ of 
$N$ active gpc coefficients in the parametric system's 
forward response.
} 
We have indicated that several algorithms for the 
linear complexity computation of approximate parametrizations
including prediction of the sets $\Lambda_N$ 
with quasi optimality properties (in the sense of best $N$-term
approximations) are now available.

In \cite{AScSt}, based on the present work, 
we present a detailed 
analysis including the error incurred through Finite Element
discretization of the forward problem
in the physical domain $D$, under slightly
stronger hypotheses on the data $u$ and $f$ than studied here.
Implementing these methods, and comparing them
with other methods such as those studied in \cite{HaSS12},
will provide further gudiance for the development of
the promising ideas introduced in this paper, and
variants on them.

Furthermore, we have assumed in the present paper that the 
observation functional $\cO(\cdot)\in V^*$ which
precludes, in space dimensions $2$ and higher, 
point observations. Once again, results which are 
completely analogous to those in the present paper
hold also for such $\cO$, albeit again under stronger
hypotheses on $u$ and on $f$.
This will also be elaborated on in \cite{AScSt}.

As indicated in 
\cite{CCDS11,CDS1,CDS2,CS+CJG11,BAS09,CJG11}
the gpc parametrizations (by either Taylor- or Legendre type
polynomial chaos representations)
of the laws of these quantities 
allow a choice of discretization of each 
gpc coeffcient of the quantity of interest
by sparse tensorization of hierarchic bases in the physical
domain $D$ and the gpc basis functions $L_\nu(y)$ resp. $y^\nu$
so that the additional discretization error incurred by 
the discretization in $D$ can be kept of the order of the
gpc truncation error with an overall computational
complexity which does not exceed that of a single, deterministic
solve of the forward problem.
These issues will be addressed in \cite{AScSt} as well.

\ack
CS is supported by SNF and by ERC under FP7 Grant AdG247277. 
AMS is supported by EPSRC and ERC.

\section*{References}

\bibliographystyle{iopart-num}

\end{document}